\documentclass[12pt]{amsart}%
\usepackage{graphicx}
\usepackage{amsmath}
\usepackage{fullpage}
\usepackage[notref,final]{showkeys}
\usepackage{amsfonts}
\usepackage{amssymb}%
\providecommand{\U}[1]{\protect\rule{.1in}{.1in}}
\newtheorem{theorem}{Theorem}

\newtheorem{corollary}[theorem]{Corollary}

\newtheorem{lemma}[theorem]{Lemma}

\newtheorem{proposition}[theorem]{Proposition}
\newtheorem{remark}[theorem]{Remark}

\renewenvironment{proof}[1][Proof]{\textbf{#1.} }{\ \rule{0.5em}{0.5em}}
\numberwithin{theorem}{section} \numberwithin{equation}{section}
\begin{document}
\title{ On transitive operator algebras in real Banach spaces}
\author{Edward Kissin, Victor S. Shulman and Yurii V. Turovskii}
\dedicatory{In memory of Victor Lomonosov}\maketitle

\begin{abstract}
We consider weakly closed transitive algebras of operators containing non-zero
compact operators in real Banach spaces (Lomonosov algebras). It is shown that
they are naturally divided in three classes: the algebras of real, complex and
quaternion classes. The properties and characterizations of algebras in each
class as well as some useful examples are presented. It is shown that in
separable real Hilbert spaces there is a continuum of pairwise non-similar
Lomonosov algebras of complex type and of quaternion type.

\end{abstract}

\section{ Introduction}

Almost fifty years ago Victor Lomonosov \cite{L} proved that any algebra of
operators on a complex Banach space $\mathcal{X}$ that contains a non-zero
compact operator either has a non-trivial closed invariant subspace or is
dense in the algebra $\mathcal{B}(\mathcal{X})$ of all operators with respect
to the strong operator topology (SOT) (or equivalently weak operator
topology). Thus the only (SOT)-closed transitive (= having only trivial
invariant subspaces) operator algebra containing a non-zero compact operator
is $\mathcal{B}(\mathcal{X})$ itself. We call (SOT)-closed transitive algebras
containing non-zero compact operators in a (real or complex) Banach space
\textit{Lomonosov algebras}. Then the Lomonosov result states that in any
complex Banach space there is only one Lomonosov algebra. In this paper we
establish that the structure of Lomonosov algebras in real spaces is much more
complicated, pithy and intriguing.

The considerable interest in the theory of invariant subspaces in real Banach
spaces that arose in recent times can be partially explained by its relations
to infinite-dimensional extremal problems and representation theory (see for
example the work of Atzmon \cite{Atzmon} and references therein). Another
reason is the fact that some interesting technical tools developed in this
theory allowed to solve several classical problems which are still open for
complex spaces. As a revealing example one can mention the theorem of
Simoni\v{c} \cite{Sim} on the existence of an invariant subspace for an
operator with compact imaginary part (see also the earlier result of Lomonosov
\cite{Lreal} and the work of Lomonosov and Shulman \cite{LS1}, where the
commutative families of such operators and their analogues for Banach spaces
were considered). The situation with the Lomonosov algebras is opposite: in
real spaces their study is more difficult than in complex spaces.

Even if $\mathcal{X}$ is finite-dimensional, the list of Lomonosov algebras is
exhausted by the algebra $\mathcal{B}(\mathcal{X})$ only in the case of odd
$\dim\mathcal{X}$. If $\dim\mathcal{X}=2n$ then the algebra $M_{n}%
(\mathbb{C})$ of all complex $n\times n$ matrices is isomorphic to the
subalgebra
\[
\mathcal{A}=\left\{
\begin{pmatrix}
T & -R\\
R & T
\end{pmatrix}
\text{: }T,R\in M_{n}(\mathbb{R})\right\}  \text{ of }B(\mathcal{X}%
)=M_{2n}(\mathbb{R}).
\]
It is transitive, so that $\mathcal{A}$ is a Lomonosov algebra. Similarly, if
$\dim\mathcal{X}=4n$ then the algebra $\ M_{n}(\mathbb{H})$ of all quaternion
$n\times n$ matrices is isomorphic to a transitive subalgebra of the algebra
$B(\mathcal{X})=M_{4n}(\mathbb{R})$. So this is another example of a Lomonosov
algebra. Moreover, every finite-dimensional Lomonosov algebra is isomorphic
either to $M_{n}(\mathbb{R}),$ or to $M_{n}(\mathbb{C}),$ or to $M_{n}%
(\mathbb{H})$, for some $n.$

It was natural to expect that in infinite dimensional spaces there exist a
rich variety of Lomonosov algebras, but no attempt has been made to describe
them or to give new meaningful examples up to now. Our aim is to fill in this
gap: to analyze the general properties and present non-trivial classes of
examples of Lomonosov algebras in infinite dimensional real Banach spaces.

In Section 2 we study not necessarily closed transitive algebras containing
non-zero finite rank operators (\textit{$L$-algebras}, for short) and obtain
for them a general density type theorem. It states that if $\mathcal{A}$ is an
$L$-algebra then, given a family $y_{1},...,y_{n}\in X$, $\varepsilon>0$ and
linearly independent set $W\subset\mathcal{X}$ of cardinality $4n,$ one can
find $x_{1},...,x_{n}\in W$ and $T\in\mathcal{A}$ with $\Vert Tx_{i}%
-y_{i}\Vert<\varepsilon$. We show that $L$-algebras are naturally divided in
three classes: the algebras of \textit{real}, \textit{complex} and
\textit{quaternion} \textit{types}, and that for each class the density
theorem can be given in a more informative form. A consequence of these
results is the fact that the only Lomonosov algebra of real type is the
algebra $\mathcal{B}(\mathcal{X})$. We deduce also some other
characterizations of algebras of complex and quaternion types. For example we
show that an algebra $\mathcal{A}$ belongs to the quaternion type if and only
if the rank of each finite-rank operator in $\mathcal{A}$ is a multiple of 4.
This allows us to show that each $L$-algebra is contained in a maximal
$L$-algebra of the same type.

In Section 3 we consider Lomonosov algebras of complex type. To each closed
operator $S$ satisfying condition $S^{-1}=-S,$ we relate a Lomonosov algebra
$\mathcal{A}_{S}$ of complex type. We study the properties of such algebras
and show that they form a dominating class: Each Lomonosov algebra of complex
type is contained in some algebra $\mathcal{A}_{S}$. We prove a locality
theorem for such algebras: Each algebra $\mathcal{A}_{S}$ is the closure of
its ideal of finite rank operators. Then, constructing a special system of
algebras $\mathcal{A}_{S}$ and using the technique of operator ranges, we
prove that in a separable infinite-dimensional real Hilbert space there is a
continuum of pairwise non-similar Lomonosov algebras of complex type.

In Section 4 similar results are obtained for the algebras of quaternion type.
However, as the quaternion case is more complicated than the complex case, our
approach differs. We introduce and study the notion of a closed representation
of a finite group on a Banach space. Then, applying the results to the
quaternion group $G_{\mathbb{H}},$ we construct Lomonosov algebras
$\mathcal{A}_{\pi}\subset\mathcal{B}(\mathcal{X})$ starting with any closed
and regular representations $\pi$ of $G_{\mathbb{H}}$ on $\mathcal{X}$. Using
this construction we show that in a real separable, infinite-dimensional
Hilbert space there is a continuum of pairwise non-similar Lomonosov algebras
of quaternion type.

It should be underlined that our picture of the variety and the structure of
Lomonosov algebras is still very far from completeness. It suffices to say
that we have no examples of non-unital Lomonosov algebras and we do not even
know if they exist (see the discussion of some open problems at the end of the
paper). We hope that this subject will be further developed in subsequent studies.

\section{ Density theorems}

In what follows, unless stated otherwise, by $\mathcal{X}$ we denote an
infinite-dimensional real Banach space. Here we consider transitive algebras
of operators on $\mathcal{X}$ that contain a non-zero finite rank operator;
for brevity we call them $L$-\textit{algebras}.

Any Lomonosov algebra is an $L$-algebra: the presence of non-zero finite-rank
operators (moreover, projections) in Lomonosov algebras was proved in
\cite[Lemma 5.5]{LRT}. Unfortunately the proof in \cite{LRT} does not seem to
be very transparent; anyway since we need somewhat less we may give a shorter
and simpler proof.

\begin{lemma}
\label{ComFin} Let $\mathcal{X}$ be a real Banach space and $T$ be a compact
operator on $\mathcal{X}$ that has a non-zero eigenvalue $\lambda$. Then the
norm-closed algebra $\mathcal{A}=\mathcal{A}(T)$ generated by $T$ contains a
non-zero finite rank operator.
\end{lemma}

\begin{proof}
It is well known that the statement holds for complex spaces: the Riesz
projection $P$ corresponding to $\{\lambda\}$ is a limit of a sequence of
polynomials of $T$. Let $\mathcal{Z}$ be a complexification of $\mathcal{X}$,
i.e. $\mathcal{Z}=\mathcal{X}\oplus\mathcal{X}$ and $\mathcal{X}$ is included
into $\mathcal{Z}$ by identification of $x\in\mathcal{X}$ with $x\oplus
0\in\mathcal{Z}$. We denote by $\mathcal{I}$ the operator on $\mathcal{Z}$
acting by the rule
\[
\mathcal{I}(x\oplus y)=(-y)\oplus x,\text{ or in matrix form }\mathcal{I}=%
\begin{pmatrix}
0 & -1\\
1 & 0
\end{pmatrix}
\]
One can consider $\mathcal{Z}$ as a complex space setting $(\alpha
+i\beta)z=(\alpha1+\beta\mathcal{I})z$. For each operator $K$ on $\mathcal{X}%
$, the operator $K\oplus K$ on $\mathcal{Z}$ is complex-linear. The operator
$T\oplus T$ is compact and $\lambda$ is its eigenvalue. Let $p_{n}%
(t)=\sum_{k=1}^{N_{n}}(\alpha_{nk}+i\beta_{nk})t^{k}$ be a sequence of
polynomials with complex coefficients such that $p_{n}(t)$ tends to a non-zero
finite rank operator $W$. Then
\[
p_{n}(T\oplus T)=\sum_{k=1}^{N_{n}}(\alpha_{nk}+i\beta_{nk})%
\begin{pmatrix}
T^{k} & 0\\
0 & T^{k}%
\end{pmatrix}
=\sum_{k=1}^{N_{n}}%
\begin{pmatrix}
\alpha_{nk}T^{k} & -\beta_{nk}T^{k}\\
\beta_{nk}T^{k} & \alpha_{nk}T^{k}%
\end{pmatrix}
=%
\begin{pmatrix}
q_{n}(T) & r_{n}(T)\\
-r_{n}(T) & q_{n}(T)
\end{pmatrix}
,
\]
where $q_{n}$ and $r_{n}$ are polynomials with real coefficients. It follows
that
\[
W=%
\begin{pmatrix}
W_{1} & -W_{2}\\
W_{2} & W_{1}%
\end{pmatrix}
\]
and $q_{n}(T)\rightarrow W_{1}$, $r_{n}(T)\rightarrow W_{2}$. Since at least
one of operators $W_{1},W_{2}$ is non-zero, $\mathcal{A}(T)$ contains a
non-zero finite rank operator.
\end{proof}

\medskip

The famous Lomonosov Lemma \cite{L} establishes the presence of compact
operators with non-zero spectra in transitive algebras containing compact
operators; it is easy to see that its proof in \cite{L} works for real spaces
as well as for complex ones. Applying Lemma \ref{ComFin} we obtain the needed result:

\begin{corollary}
\label{LomAlg} Every transitive norm-closed algebra of operators on a real
Banach space containing a non-zero compact operator contains a non-zero finite
rank operator.
\end{corollary}

So Lomonosov algebras can be equivalently defined as (SOT)-closed $L$-algebras.

Recall that an algebra $\mathcal{A}$ of operators on a linear space
$\mathcal{X}$ over a field or, more generally, over a division ring is called
\textit{strictly transitive}, if for any $x,y\in X$ with $x\neq0$, there is
$T\in\mathcal{A}$ with $Tx=y$ . Furthermore $\mathcal{A}$ is called
\textit{strictly dense} if for each finite linearly independent family
$(x_{1},...,x_{n})\subset X$ and any family $(y_{1},...,y_{n})\subset
\mathcal{X}$, there is $T\in\mathcal{A}$ with $Tx_{i} = y_{i}$, $i = 1,...,n$.

Let $k\in N$. Let us say that $\mathcal{A}$ is \textit{strictly ($1/k$%
)-dense}, if for every linearly independent family $(x_{1},...,x_{kn}%
)\subset\mathcal{X}$ and any family $(y_{1},...,y_{n})\subset\mathcal{X}$,
there are $j_{1},...,j_{n} \in\{1,...,kn\}$ and $T\in\mathcal{A}$ with
$\{y_{i}: 1\le i\le n\}\subset\{Tx_{j}: j = 1,...,kn\}$. In other words $y_{i}
= Tx_{j_{i}}$, $1\le i\le n$. \bigskip

In the operator theory setting an algebra $\mathcal{A}$ of bounded linear
operators on a Banach space $\mathcal{X}$ is called

\textit{transitive} if for any $\varepsilon>0$ and any $x,y\in\mathcal{X}$
with $x\neq0$, there is $T\in\mathcal{A}$ with $\|Tx-y\| < \varepsilon$ (it is
easy to check that this definition is equivalent to one given above, namely
the absence of non-trivial closed invariant subspaces),

\textit{dense} if for any $\varepsilon>0$, and for each finite linearly
independent family $(x_{1},...,x_{n})\subset\mathcal{X}$ and any family
$(y_{1},...,y_{n})\subset\mathcal{X}$, there is $T\in\mathcal{A}$ with
$\|Tx_{i} - y_{i}\| < \varepsilon$, $i = 1,...,n$.

($1/k$)\textit{-dense}, if for any $\varepsilon>0$, every linearly independent
family $(x_{1},...,x_{kn})\subset\mathcal{X}$ and any family $(y_{1}%
,...,y_{n})\subset\mathcal{X}$, there are $j_{1},...,j_{n}\in\{1,...,kn\}$ and
$T\in\mathcal{A}$ with $\Vert y_{i}-Tx_{j_{i}}\Vert<\varepsilon$, $1\leq i\leq
n$.

Note that if $\mathcal{A}$ is ($1/k$)-dense then it is ($1/(k+1)$)-dense.
\bigskip

In this section our aim is to establish the following general version of
density theorem for $L$-algebras.

\begin{theorem}
\label{hT} Every $L$-algebra of operators in a real Banach space is
\emph{(1/4)}-dense.
\end{theorem}

This theorem will be deduced from more strong and detailed results. Before
proving them we introduce some notation and establish several auxiliary results.

Let $\mathcal{F}(\mathcal{X})$ be the ideal of all finite-rank operators on
$\mathcal{X}$. For an $L$-algebra $\mathcal{A}$ on $\mathcal{X}$, we set
\[
\mathcal{A}^{F}=\mathcal{A}\cap\mathcal{F}(\mathcal{X})\text{ and }%
\mathcal{X}^{F}=\mathcal{A}^{F}\mathcal{X}:=\mathrm{lin}\{Tx:T\in
\mathcal{A}^{F},x\in\mathcal{X}\}.
\]
Since $\mathcal{A}^{F}$ is an ideal of $\mathcal{A}$, the linear subspace
$\mathcal{X}^{F}$ is invariant for $\mathcal{A}$ and therefore dense in
$\mathcal{X}$.

\begin{lemma}
\label{str} The restriction of $\mathcal{A}^{F}$ to $\mathcal{X}^{F}$ is a
strictly transitive algebra of operators on $\mathcal{X}^{F}$.
\end{lemma}

\begin{proof}
Let $0\neq x_{0}\in\mathcal{X}^{F}$ and $\mathcal{Y}=\mathcal{A}^{F}x_{0}$.
Then $\mathcal{Y}$ is $\mathcal{A}$-invariant and therefore dense in
$\mathcal{X}$. It follows that for each $T\in\mathcal{A}^{F}$, $T\mathcal{Y}$
is dense in $T\mathcal{X}$. Since $\dim T\mathcal{X}<\infty$, $T\mathcal{Y}%
=T\mathcal{X}$. Therefore $T\mathcal{X}\subset\mathcal{Y}$. Since
$\mathcal{X}^{F}=\mathrm{lin}(\cup_{T\in\mathcal{A}^{F}}T\mathcal{X})$ we see
that $\mathcal{X}^{F}\subset\mathcal{Y}$. The converse inclusion follows from
the definition of $\mathcal{X}^{F}$. We proved that $\mathcal{Y}%
=\mathcal{X}^{F}$, so $\mathcal{A}^{F}$ is strictly transitive.\bigskip
\end{proof}

Let $\mathbb{D}$ be the algebra of all linear operators on $\mathcal{X}^{F}$
commuting with $\mathcal{A}^{F}$. Since $\mathcal{A}^{F}$ is strictly
transitive each non-zero operator in $\mathbb{D}$ is invertible (its kernel
and range are invariant for $\mathcal{A}^{F}$). Thus $\mathbb{D}$ is a
division algebra. By Jacobson's Density Theorem (see \cite[Section 5]{Bres}
for several convenient formulations), we obtain the following result.

\begin{corollary}
\label{Jac} $\mathcal{A}^{F}$ is strictly dense on $\mathcal{X}^{F}$, which is
considered as a linear space over $\mathbb{D}$.
\end{corollary}

\begin{lemma}
\label{finDim} The $\mathbb{R}$-algebra $\mathbb{D}$ is finite-dimensional.
\end{lemma}

\begin{proof}
The range $\mathcal{M}=T\mathcal{X}^{F}$ of a non-zero operator $T\in
\mathcal{A}^{F}$ is invariant for $\mathbb{D}$. The map $S\mapsto
S|_{\mathcal{M}}$ from $\mathbb{D}$ to the algebra $\mathcal{L}(\mathcal{M})$
of all linear operators on $\mathcal{M}$ is injective. Indeed if
$S|_{\mathcal{M}}=0$, then $\ker S\neq0$, but $\ker S$ must be trivial being
an invariant subspace for $\mathcal{A}^{F}$. So $\dim(\mathbb{D})\leq
\dim(\mathcal{L}(\mathcal{M}))<\infty$. \bigskip
\end{proof}

Applying the famous Frobenius's Theorem on finite-dimensional division
algebras over $\mathbb{R}$, we conclude that only 3 cases are possible:

a) $\dim(\mathbb{D}) = 1$, $\mathbb{D}$ is isomorphic to $\mathbb{R}$; in this
case we say that $\mathcal{A}$ is an algebra \textit{of real type},

b) $\dim(\mathbb{D}) = 2$, $\mathbb{D}$ is isomorphic to $\mathbb{C}$;
$\mathcal{A}$ is an algebra \textit{of complex type},

c) $\dim(\mathbb{D})=4$, $\mathbb{D}$ is isomorphic to the algebra of
quaternions, $\mathbb{H}$; $\mathcal{A}$ is an algebra \textit{of quaternion
type}.

\begin{lemma}
\label{more} Each operator $T\in\mathcal{A}$ is $\mathbb{D}$-linear on
$\mathcal{X}^{F}$, i.e., $T$ commutes with each $\Lambda\in\mathbb{D}$ on
$\mathcal{X}^{F}$.
\end{lemma}

\begin{proof}
Let $x\in\mathcal{X}^{F}$. Since $\mathcal{A}^{F}$ is strictly transitive on
$\mathcal{X}^{F}$, there is $K\in\mathcal{A}^{F}$ with $Kx=x$. So for each
$\Lambda\in\mathbb{D}$, $T\Lambda x = T\Lambda Kx = TK\Lambda x = \Lambda(TK)
x = \Lambda Tx$. Thus $T$ commutes with $\mathbb{D}$ on $\mathcal{X}^{F}$.
\end{proof}

\begin{lemma}
\label{indep} For every linearly independent set $\{u_{1},...,u_{m}%
\}\subset\mathcal{X}$, there is $K\in\mathcal{A}^{F}$ such that the set
$\{Ku_{1},...,Ku_{m}\}$ is linearly independent.
\end{lemma}

\begin{proof}
For $m=1$ this is evident. Using induction, assume that for $m$-element sets
the statement is true. Aiming at a contradiction, suppose that for a system
\{$u_{1},...,u_{m+1}\}$ the statement fails. Then for each $K\in
\mathcal{A}^{F}$ satisfying the condition
\begin{equation}
\text{ the set }\{Ku_{1},Ku_{2},...,Ku_{m}\}\text{ is linearly independent},
\label{indep1}%
\end{equation}
there are numbers $t_{i}(K)\in\mathbb{R}$, $1\leq i\leq m$, such that
\begin{equation}
Ku_{m+1}=\sum_{i=1}^{m}t_{i}(K)Ku_{i}. \label{indeP}%
\end{equation}
Fix $K_{0}\in\mathcal{A}^{F}$ satisfying (\ref{indep1}) and set $\lambda
_{i}=t_{i}(K_{0})$, $1\leq i\leq m$. Let us denote $\mathrm{lin}%
(u_{1},...,u_{n})$ by $\mathcal{L}$ and let $\mathcal{E}$ be the set of all
operators $T\in\mathcal{A}^{F}$ satisfying the conditions
\begin{equation}
\dim T\mathcal{L}=m\text{ and }T\mathcal{L}\cap K_{0}\mathcal{L}=0.
\label{cap}%
\end{equation}
To see that $\mathcal{E}$ is non-void, let $\mathcal{M}=K_{0}\mathcal{L}$ and
$\widetilde{\mathcal{M}}$ be the $\mathbb{D}$-linear span of $\mathcal{M}$.
This is a finite-dimensional $\mathbb{D}$-linear subspace of $\mathcal{X}^{F}$
and (since $\dim_{\mathbb{D}}(\mathcal{X}^{F})=\infty$) there is a
$\mathbb{D}$-linear subspace $\mathcal{N}\subset\mathcal{X}^{F}$ with
$\widetilde{\mathcal{M}}\cap\mathcal{N}=0$ and $\dim_{\mathbb{D}}%
(\mathcal{N})=\dim_{\mathbb{D}}(\widetilde{\mathcal{M}})$. By Corollary
\ref{Jac}, $\mathcal{A}^{F}$ as a $\mathbb{D}$-algebra is strictly dense in
$\mathcal{X}^{F}$. Hence there is $R\in\mathcal{A}^{F}$ with
$R\widetilde{\mathcal{M}}=\mathcal{N}$. Clearly the operator $RK_{0}$ belongs
to $\mathcal{E}$.

We claim that $t_{i}(T) = \lambda_{i}$, $1\le i\le m$, for all $T
\in\mathcal{E}$.

The standard argument shows that $\mathcal{E}$ is open in $\mathcal{A}^{F}$:
if $T\in\mathcal{E}$ then $T+ K\in\mathcal{E}$ for all $K\in A$ with
sufficiently small norm. So fix $T\in\mathcal{E}$ and, for some sufficiently
small $\alpha$, let $\mu_{i} = t_{i}(T+\alpha K^{F})$. Then
\begin{align*}
&  \sum_{i=1}^{m}\lambda_{i}\alpha K_{0}u_{i}+\sum_{i=1}^{m}t_{i}(T)Tu_{i}%
{=}\alpha K_{0}u_{m+1}+Tu_{m+1}\\
&  =(\alpha K_{0}+T)u_{m+1}{=}\sum_{i=1}^{m}\mu_{i}(\alpha K_{0}+T)u_{i}%
=\sum_{i=1}^{m}\mu_{i}\alpha K_{0}u_{i}+\sum_{i=1}^{m}\mu_{i}Tu_{i}.
\end{align*}
whence
\[
\sum_{i=1}^{m}(\lambda_{i}-\mu_{i})\alpha K_{0}u_{i} = \sum_{i=1}^{m} (\mu_{i}
- t_{i}(T))Tu_{i}.
\]
It follows from (\ref{cap}) that both parts of the above equality are equal to
zero, whence
\[
\lambda_{i} = \mu_{i} = t_{i}(T) \text{ for } i =1,...m.
\]

So we proved that $Tu_{i+1} = \sum_{i=1}^{m}\lambda_{i}Tu_{i}$, for all
$T\in\mathcal{E}$. In other words
\[
u_{i+1} - \sum_{i=1}^{m} \lambda_{i}u_{i} \in\bigcap_{T\in\mathcal{E}}\ker T.
\]
Since $\mathcal{E}$ is open in $\mathcal{A}^{F}$, $\bigcap_{T\in\mathcal{E}%
}\ker T = \ker\mathcal{A}^{F} = \{0\}$. Therefore $u_{m+1} = \sum_{i=1}^{m}
\lambda_{i}u_{i}$, a contradiction.
\end{proof}

\begin{corollary}
\label{cons}If $\mathcal{A}^{F}$ is strictly \emph{(}$1/k$\emph{)}-dense on
$\mathcal{X}^{F}$ then $\mathcal{A}$ is \emph{(}$1/k$\emph{)}-dense.
\end{corollary}

\begin{proof}
Let $\varepsilon>0$, a linearly independent set $\{x_{1},...,x_{kn}%
\}\subset\mathcal{X}$ and a set $\{y_{1},...,y_{n}\}\subset\mathcal{X}$ be
given . Since $\mathcal{X}^{F}$ is dense in $\mathcal{X}$, there is a set
$\{z_{1},...,z_{n}\}\subset\mathcal{X}^{F}$ with $\|z_{i}-y_{i}\|<\varepsilon
$, $1\le i \le n$.

Using Lemma \ref{indep} we find $K\in\mathcal{A}^{F}$ such that the set
$\{Kx_{1},...,Kx_{kn}\}$ is linearly independent. Since $Kx_{i}\in
\mathcal{X}^{F}$, it follows from the assumption of the corollary, that there
are $P\in\mathcal{A}^{F}$ and numbers $j_{1},..,j_{n}$ with $PKx_{j_{i}}%
=z_{i}$, so it remains to set $T=PK$.\bigskip
\end{proof}

Now we can prove our "detailed"\ density theorem.

\begin{theorem}
\label{hTstrong}An L-algebra of real type is dense\emph{,} an L-algebra of
complex type is \emph{(1/2)}-dense\emph{,} an $L$-algebra of quaternion type
is \emph{(1/4)}-dense.
\end{theorem}

\begin{proof}
By Corollary \ref{cons}, it suffices to show that, for a real type algebra
$\mathcal{A}$, the algebra $\mathcal{A}^{F}$ is strictly dense on
$\mathcal{X}^{F}$; for a complex type algebra $\mathcal{A}$, the algebra
$\mathcal{A}^{F}$ is strictly (1/2)-dense on $\mathcal{X}^{F}$, and for a
quaternion type algebra $\mathcal{A}$, the algebra $\mathcal{A}^{F}$ is
strictly (1/4)-dense on $\mathcal{X}^{F}$.

The needed statement for real type algebras follows from the Jacobson Density
Theorem because in this case $\mathbb{D} = \mathbb{R}$.

Suppose that $\mathbb{D}$ is isomorphic to $\mathbb{C}$, so we may consider
$\mathcal{X}^{F}$ as a ${\mathbb{C}}$-module. Let $x_{1},...,x_{2n}$ be an
$\mathbb{R}$-linearly independent family in $\mathcal{X}^{F}$, and let
$\mathcal{M}=\mathrm{lin}_{\mathbb{C}}(x_{1},...,x_{2n})$ the $\mathbb{C}%
$-linear hull of $x_{1},...,x_{2n}$. If $\dim_{\mathbb{C}}\mathcal{M}<n$ then
$\dim_{\mathbb{R}}\mathcal{M}<2n$, a contradiction with our conditions. Thus
$\dim_{\mathbb{C}}\mathcal{M}\geq n$. Hence $\mathcal{M}$ contains a
${\mathbb{C}}$-linearly independent subfamily $x_{i_{1}},...,x_{i_{n}}$.
Therefore, by Corollary \ref{Jac}, for each $y_{1},...,y_{n}\in\mathcal{X}%
^{F}$, there is $T\in\mathcal{A}^{F}$ with $Tx_{i_{k}}=y_{k}$, $k=1,...,n$.

For algebras of quaternion type the argument is similar.\bigskip
\end{proof}

Since any dense and any (1/2)-dense algebra is (1/4)-dense, Theorem \ref{hT}
immediately follows from Theorem \ref{hTstrong}. \medskip

Below we will prove a kind of converse statement. Let us begin with a lemma
which will be repeatedly used below.

\begin{theorem}
\label{closable}For each $L$-algebra $\mathcal{A}\subset\mathcal{B}%
(\mathcal{X})$\emph{,} all operators in $\mathbb{D}$ are closable on
$\mathcal{X}$.
\end{theorem}

\begin{proof}
Let $S\in\mathbb{D}$, $x_{n}\in\mathcal{X}^{F}$, $x_{n}\rightarrow0$ and
$Sx_{n}\rightarrow y\in\mathcal{X}$. For every $T\in\mathcal{A}^{F}$,
$STx_{n}=TSx_{n}\rightarrow Ty$. On the other hand the subspace $T\mathcal{X}%
^{F}$ is invariant for $S$ and finite-dimensional so $S$ is bounded on
$T\mathcal{X}^{F}$. Since $Tx_{n}\rightarrow0$, we get that $STx_{n}%
\rightarrow0$. Therefore $Ty=0$, for each $T\in\mathcal{A}^{F}$. Since
$\cap_{T\in\mathcal{A}^{F}}\ker T=0$ (it is a closed $\mathcal{A}$-invariant
subspace), $y=0$. Thus $S$ is closable.
\end{proof}

\begin{theorem}
\label{join}\emph{(i) }An $L$-algebra $\mathcal{A}$ is dense if and only if it
is of real type.\smallskip

\emph{(ii) }An $L$-algebra $\mathcal{A}$ is \emph{(1/2)}-dense but not dense
if and only if it is of complex type.\smallskip

\emph{(iii) }An $L$-algebra $\mathcal{A}$ is \emph{(1/4)}-dense but neither
dense\emph{,} nor \emph{(1/2)}-dense if and only if it is of quaternion type.
\end{theorem}

\begin{proof}
In all cases the "only if"\; part needs to be proved.

(i) Suppose that a dense $L$-algebra $\mathcal{A}$ is not of real type. Then
there is an operator $S\in\mathbb{D}$ which is is not a scalar multiple of
$1_{\mathcal{X}^{F}}$. Therefore there is $x\in\mathcal{X}^{F}$ such that $x$
and $Sx$ are linearly independent. Choose $0\neq y\in\mathcal{X}$. Since
$\mathcal{A}$ is dense then there is a sequence $T_{n}\in\mathcal{A}$ with
$T_{n}x\to0$, $T_{n}Sx\to y$. So $ST_{n}x\to y$ which is impossible since $S$
is closable by Theorem \ref{closable}.

(ii) Using (i) we have only to show that an $(1/2)$-dense $L$-algebra cannot
be of quaternion type. Suppose that $\mathcal{A}$ is of quaternion type, so
$\dim\mathbb{D}=4$. Choose $0\neq x_{0}\in\mathcal{X}^{F}$. Then the space
$\mathcal{K}=\mathbb{D}x_{0}$ is 4-dimensional. Choose $0\neq y\in\mathcal{X}%
$. Since $\mathcal{A}$ is $(1/2)$-dense there are $x_{1},x_{2}\in\mathcal{K}$
and a sequence $T_{n}\in\mathcal{A}$ with $T_{n}x_{1}\rightarrow0$,
$T_{n}x_{2}\rightarrow y$. By definition there are $S_{1},S_{2}\in\mathbb{D}$
such that $x_{i}=S_{i}x_{0}$. Therefore $x_{2}=Sx_{1}$ where $S=S_{2}%
S_{1}^{-1}$. Then arguing as above we get $y=\lim T_{n}x_{2}=\lim T_{n}%
Sx_{1}=\lim ST_{n}x_{1}$. Since $S$ is closable, we get $y=0$, a contradiction.

(iii) Follows directly from Theorem \ref{hTstrong}.
\end{proof}

\begin{corollary}
\label{closL}An $L$-algebra $\mathcal{A}$ belongs to real\emph{,} complex or
quaternion type if and only if its (SOT)-closure $\overline{\mathcal{A}}$
belongs to the same type.
\end{corollary}

\begin{proof}
By Theorem \ref{join}, it suffices to show that for $k=1,2,4$, an algebra
$\mathcal{A}$ is $(1/k)$-dense if and only if $\overline{\mathcal{A}}$ (as
well as any algebra that lies between them) is $(1/k)$-dense. Clearly if
$\mathcal{A}$ is $(1/k)$-dense then each set containing $\mathcal{A}$ is
$(1/k)$-dense.

Conversely suppose that $\overline{\mathcal{A}}$ (or some intermediate
algebra) is $(1/k)$-dense. By definition, for any $\varepsilon>0$, for every
linearly independent family $(x_{1},...,x_{kn})\subset\mathcal{X}$ and any
family $(y_{1},...,y_{n})\subset\mathcal{X}$, there are $j_{1},...,j_{n}%
\in\{1,...,kn\}$ and $T\in\overline{\mathcal{A}}$ with $\Vert y_{i}-Tx_{j_{i}%
}\Vert<\varepsilon/2$, $1\leq i\leq n$. Now by definition of (SOT), there is
$K\in\mathcal{A}$ such that $\Vert Kx_{j_{i}}-Tx_{j_{i}}\Vert<\varepsilon/2$,
for $1\leq i\leq n$. Then $\Vert y_{i}-Kx_{j_{i}}\Vert<\varepsilon$, for
$1\leq i\leq n$. So $\mathcal{A}$ is ($1/k$)-dense.\bigskip
\end{proof}

Now we will characterize $L$-algebras of all types in terms of the ranks of
its elements. Recall that the \textit{rank}, $\mathrm{rank}(T)$, of an
operator $T$ acting on a linear space $\mathcal{M}$ is defined as the
dimension of its range $T\mathcal{M}$.

\begin{proposition}
\label{Grank}\emph{(i)} If $\mathcal{A}$ is an $L$-algebra of complex type
then $\mathrm{rank}(T)$ is even, for any $T\in\mathcal{A}^{F}$.

\emph{(ii)} If $\mathcal{A}$ is an $L$-algebra of quaternion type then
$\mathrm{rank}(T)$ is divisible by four, for any $T\in\mathcal{A}^{F}$.
\end{proposition}

\begin{proof}
(i) If $\mathcal{A}$ is an $L$-algebra of complex type then $\mathcal{X}^{F}$
has the structure of complex space and all operators in $\mathcal{A}^{F}$ are
$\mathbb{C}$-linear. In particular the range of an operator $T\in
\mathcal{A}^{F}$ is a complex subspace of $\mathcal{X}^{F}$ and its dimension
equals its complex dimension multiplied by 2.

(ii) Similarly, the dimension of an $\mathbb{H}$-linear subspace equals its
quaternion dimension multiplied by 4.\bigskip
\end{proof}

For any $L$-algebra $\mathcal{A}$, we set
\[
r(\mathcal{A}) = \min\{\mathrm{rank}(T): T\in\mathcal{A}^{F}\}.
\]

\begin{theorem}
\label{ranks} An $L$-algebra $\mathcal{A}$ is of real, complex or quaternion
type if and only if $r(\mathcal{A})$ \emph{= 1, 2} or \emph{4,} respectively.
\end{theorem}

\begin{proof}
Let $T_{0}$ be an operator of minimal rank in $\mathcal{A}$. We denote by
$\mathcal{Z}$ the subspace $T_{0}\mathcal{X}^{F}$ of $\mathcal{X}^{F}$ and let
$z_{1},...,z_{m}$ be a basis of $\mathcal{Z}$ as a linear space over
$\mathbb{D}$. Our aim is to show that $m=1$.

We claim that the algebra $\mathcal{B} = T_{0}\mathcal{A}|_{\mathcal{Z}} =
\{T_{0}K|_{\mathcal{Z}}: K\in\mathcal{A}^{F}\}$ coincides with the algebra of
all $\mathbb{D}$-linear operators on ${\mathcal{Z}}$.

Indeed, let $V$ be an arbitrary $\mathbb{D}$-linear operator on $\mathcal{Z}$
and $w_{i}=Vz_{i}$, $1\leq i\leq m$. Since $\mathcal{Z}=T_{0}\mathcal{X}^{F}$
there are vectors $x_{1},...,x_{m}\in\mathcal{X}^{F}$ such that $w_{i}%
=T_{0}x_{i}$, for all $i$. By Corollary \ref{Jac}, there is $K\in\mathcal{A}$
with $Kz_{i}=x_{i}$, whence the operator $Q=T_{0}K|_{\mathcal{Z}}%
\in\mathcal{B}$ satisfies the condition $Qz_{i}=Vz_{i}$, $1\leq i\leq m$, and
therefore coincides with $V$.

It follows that $\mathcal{B}$ contains an operator $T_{1} = T_{0}%
K_{1}|_{\mathcal{Z}}$ whose range as a $\mathbb{D}$-linear space has dimension
1. Clearly the operator $T_{0}K_{1}T_{0}\in\mathcal{A}$ has the same range as
$T_{1}$. Thus, by the choice of $T_{0}$, the $\mathbb{D}$-rank of $T_{0}$
equals 1. So the $\mathbb{R}$-rank of $T_{0}$ is 1, 2 or 4, if $\mathcal{A}$
is of real, complex or quaternion type respectively.
\end{proof}

\begin{corollary}
\label{maxim} Each $L$-algebra is contained in a maximal $L$-algebra of the
same type.
\end{corollary}

\begin{proof}
By Zorn's lemma it suffices to show that the union of a linearly ordered
family of algebras of given type belongs to this type. Let $\mathcal{A}%
=\cup_{\alpha\in\Lambda}\mathcal{A}_{\alpha}$, where the family $(\mathcal{A}%
_{\alpha})_{\alpha\in\Lambda}$ is linearly ordered (or up-directed). If all
$\mathcal{A}_{\alpha}$ have the same type, then by Theorem \ref{ranks} they
have the same minimal rank: $r(\mathcal{A}_{\alpha})=k$, for all $\alpha$. It
follows that $r(\mathcal{A})=k$, because each operator in $\mathcal{A}$
belongs to some $\mathcal{A}_{\alpha}$. Again applying Theorem \ref{ranks} we
conclude that type of $\mathcal{A}$ is the same as the type of $\mathcal{A}%
_{\alpha}$.
\end{proof}

\begin{remark}
\label{maxalg} \emph{ It follows from Corollary \ref{closL} that maximal
algebras of a given type are Lomonosov algebras.}
\end{remark}

\section{ Lomonosov algebras of complex type}

It follows from Theorem \ref{hTstrong} that Lomonosov algebras of real type
are dense and therefore can be described trivially:
\[
\mathcal{A} = \mathcal{B}(\mathcal{X})
\]
because the density of an algebra $\mathcal{A}$ is equivalent to the condition
that the (SOT)-closure of $\mathcal{A}$ in $\mathcal{B}(X)$ coincides with
$\mathcal{B}(X)$.

Passing on to algebras of complex type, we begin with discussion of some examples.

We call by \textit{a partial complex structure on $\mathcal{X}$} (PCS, for
brevity) an operator $S$ defined on a dense subspace $\mathcal{D}%
(S)\subset\mathcal{X}$ and satisfying the condition
\begin{equation}
S^{-1}=-S, \label{3.2}%
\end{equation}
which implies the equality $S\mathcal{D}(S)=\mathcal{D}(S)$.

One of the ways to construct a PCS is to choose a sequence of linearly
independent 2-dimensional subspaces $\mathcal{X}_{n}$ and operators $S_{n}$ on
$\mathcal{X}_{n}$ satisfying the condition $\mathcal{S}_{n}^{2}+\mathbf{1}%
_{\mathcal{X}_{n}}=0$ (they have matrices of the form
\[%
\begin{pmatrix}
0 & -1\\
1 & 0
\end{pmatrix}
\]
in appropriate bases. Then we set $\mathcal{D}=\mathrm{lin}(\cup
_{n}\mathcal{X}_{n})$ and define $S$ on $\mathcal{D}$ as the direct sum of
$S_{n}$. The properties of such operators depend on the choice of subspaces
$\mathcal{X}_{n}$ and the bases in $\mathcal{X}_{n}$. For example, if
$\mathcal{X}$ is a Hilbert space one can take mutually orthogonal
$\mathcal{X}_{n}$ and orthonormal bases in each $\mathcal{X}_{n}$, then $S$
will be bounded. If the bases in $\mathcal{X}_{n}$ is such that $\Vert
S_{n}\Vert\rightarrow\infty$ then $S$ will be closable but not bounded.

It is important that the closure of a closable PCS is a PCS, as the following
lemma shows.

\begin{lemma}
\label{closPCS} Let $\mathcal{X}$ be a Banach space\emph{,} $\mathcal{X}_{0}$
its dense subspace\emph{,} and let $S$\emph{: }$\mathcal{X}_{0}\rightarrow
\mathcal{X}_{0}$ be a closable bijective operator satisfying the condition
$S^{-1}=-S$. Then the closure $\overline{S}$ of $S$ is a closed bijective
operator satisfying the condition $(\overline{S})^{-1}=-\overline{S}$.
\end{lemma}

\begin{proof}
Let $x\in\mathcal{D}(\overline{S})$ and $y=\overline{S}x$. Then there is a
sequence $x_{n}\in\mathcal{X}_{0}$ with $x_{n}\rightarrow x$, $Sx_{n}%
\rightarrow y$. Since $x_{n}=-S(Sx_{n})$, we see that $y\in\mathcal{D}%
(\overline{S})$ (so $\overline{S}(\mathcal{D}(\overline{S}))\subset
\mathcal{D}(\overline{S})$) and $x=-\overline{S}y$ (so $\mathcal{D}%
(\overline{S})\subset\overline{S}(\mathcal{D}(\overline{S}))$. Thus
$\overline{S}(\mathcal{D}(\overline{S}))=\mathcal{D}(\overline{S})$ and
$\overline{S}^{2}x=-x$, for all $x\in\mathcal{D}(\overline{S})$. Hence
$(\overline{S})^{-1}=-\overline{S}$ on $\mathcal{D}(\overline{S})$.\bigskip
\end{proof}

Recall that for any operator $W$ defined on a linear subspace $\mathcal{D}(W)$
of a Banach space $\mathcal{X}$, its graph $G_{W}=\{(x,Wx):x\in\mathcal{D}%
(W)\}$ is a linear subspace of the Banach space $\mathcal{X\oplus X}$. We
identify the space $\mathcal{X}^{\ast}\oplus\mathcal{X}^{\ast}$ with the dual
space of $\mathcal{X\oplus X}$, setting
\[
(f_{1}\oplus f_{2})(x_{1}\oplus x_{2})=f_{1}(x_{2})-f_{2}(x_{1}).
\]
If $\mathcal{D}(W)$ is dense in $\mathcal{X}$ then the annihilator
$G_{W}^{\bot}$ of $G_{W}$ in $\mathcal{X}^{\ast}\oplus\mathcal{X}^{\ast}$ does
not contain pairs of the form $(0,f)$, $f\neq0$, and therefore there is an
operator $V$ defined on a subspace $\mathcal{D}_{V}\subset\mathcal{X}^{\ast}$
such that $G_{W}^{\bot}=\{(f,Vf)$: $f\in\mathcal{D}_{V}\}$. The operator $V$
is called the adjoint of $W$ and denoted by $W^{\ast}$. The domain
$\mathcal{D}(W^{\ast})$ of $W^{\ast}$ can be described as follows
\begin{equation}
\mathcal{D}(W^{\ast})=\{f\in\mathcal{X}^{\ast}\text{: there is }C>0\text{ such
that }|f(Wx)|\leq C\Vert x\Vert\text{ for all }x\in\mathcal{D}(W)\}.
\label{sopr}%
\end{equation}
If $W$ is closable then $\mathcal{D}_{W^{\ast}}$ is weak*-dense in
$\mathcal{X}^{\ast}$. If, moreover, $W$ is closed then $G_{W}$ is closed and
therefore (see for example \cite[Theorem 4.4.6]{Rud}), $(G_{W}^{\bot})^{\bot
}\cap(\mathcal{X}\oplus\mathcal{X})=G_{W}$. In other words,
\begin{equation}
\text{if }x,y\in\mathcal{X}\text{ and }f(x)=(W^{\ast}f)(y),\text{ for all
}f\in\mathcal{D}(W^{\ast}),\text{ then }y\in\mathcal{D}(W),\text{ }x=Wy.
\label{dualpair}%
\end{equation}

In particular, if $S$ is a closable PCS, then the adjoint operator $S^{\ast}$
is defined on a weak*-dense subspace $\mathcal{D}(S^{\ast})\subset
\mathcal{X}^{\ast}$. If $f\in\mathcal{D}(S^{\ast})$ and $g=S^{\ast}f$ then
$g\in\mathcal{D}(S^{\ast})$ and $S^{\ast}g=-f$. Indeed,%
\[
Sx\in\mathcal{D}(S)\text{ by (\ref{3.2}), and }g(Sx)=S^{\ast}%
f(Sx)\overset{(\ref{dualpair})}{=}f(SSx)=-f(x),\text{ for }x\in\mathcal{D}%
(S),
\]
whence $|g(Sx)|\leq\Vert f\Vert\Vert x\Vert$ and the claim follows from
(\ref{sopr}).

For any $v\in\mathcal{D}(S)$ and $f\in\mathcal{D}(S^{\ast})$, we set
\begin{equation}
T_{v,f}=v\otimes f-Sv\otimes S^{\ast}f. \label{2ten}%
\end{equation}
It is obvious that on $\mathcal{D}(S)$ the operator $T_{v,f}$ acts by the
rule
\[
T_{v,f}x=f(x)v-f(Sx)Sv.
\]
Clearly, $T_{v,f}$ is a rank 2 operator that maps $\mathcal{X}$ to
$\mathcal{D}$. Let us show that it commutes with $S$ on $\mathcal{D}(S)$.
Indeed,
\[
T_{v,f}Sx=f(Sx)v-f(S^{2}x)Sv=-f(Sx)S^{2}v+f(x)Sv=S(-f(Sx)Sv+f(x)v)=ST_{v,f}x,
\]
for $x\in\mathcal{D}(S)$.

For any PCS $S$, we set
\[
\mathcal{A}_{S}=\{T\in\mathcal{B}(\mathcal{X}):TS\subset ST\}.
\]

\begin{theorem}
\label{comType}If a partial complex structure $S$ on a real Banach space
$\mathcal{X}$ is closed\emph{,} then $\mathcal{A}_{S}$ is a Lomonosov algebra
of complex type on $\mathcal{X}$.
\end{theorem}

\begin{proof}
Clearly, $\mathcal{A}_{S}$ is an algebra. To see that $\mathcal{A}_{S}$ is
(SOT)-closed let $T_{n}\in\mathcal{A}_{S}$ and $T_{n}\overset{sot}{\rightarrow
}T\in\mathcal{B}(\mathcal{X}).$ Then, for each $x\in\mathcal{D}(S)$,
$T_{n}x\rightarrow Tx$ and $ST_{n}x=T_{n}Sx\rightarrow TSx$. Since $S$ is
closed, $Tx\in\mathcal{D}(S)$ and $STx=TSx$. This means that $TS\subset ST$,
$T\in\mathcal{A}_{S}$.

Let $x_{0}\in\mathcal{X}$, $y_{0}\in\mathcal{X}$ and $\varepsilon>0$. Our aim
is to find an operator $T\in\mathcal{A}_{S}$ with $\Vert Tx_{0}-y_{0}%
\Vert<\varepsilon$. In fact, we will find such $T$ in $\mathcal{A}_{S}%
^{F}:=\mathcal{A}_{S}\cap\mathcal{F}(\mathcal{X})$.

Let us show that there is $f_{0}\in\mathcal{D}(S^{\ast})$ with $f_{0}%
(x_{0})=1$ and $(S^{\ast}f_{0})(x_{0})=0$. Indeed, it suffices to find
$f\in\mathcal{D}(S^{\ast})$ with $(S^{\ast}f)(x_{0})=0$ and $f(x_{0})\neq0$.
If this is impossible then considering functionals $F_{1}(f)=(S^{\ast}%
f)(x_{0})$, $F_{2}(f)=f(x_{0})$ on $\mathcal{D}(S^{\ast})$ we have that $\ker
F_{1}\subset\ker F_{2}$. Therefore there is $\alpha\in\mathbb{R}$ with
$F_{2}=\alpha F_{1}$. Thus $f(x_{0})=\alpha(S^{\ast}f)(x_{0})$, for all
$f\in\mathcal{D}(S^{\ast})$.
By (\ref{dualpair}), $\alpha x_{0}\in\mathcal{D}(S)$ and $x_{0}=S(\alpha
x_{0})$. However, $S$ has no eigenvalues (since $S^{2}\subset-\mathbf{1}%
_{\mathcal{X}}$). This contradiction proves our claim.

Now we choose $u\in\mathcal{D}(S)$ with $\Vert y_{0}-u\Vert<\varepsilon$ and
set $T=T_{u,f_{0}}$. Then it follows from (\ref{2ten}) and the definition of
$f_{0}$ that
\[
Tx_{0}=f_{0}(x_{0})u-(S^{\ast}f_{0})(x_{0})Su=u.
\]
Thus $\Vert Tx_{0}-y_{0}\Vert=\Vert u-y_{0}\Vert<\varepsilon$. This means that
the algebra $\mathcal{A}_{S}$ is transitive. Since $T$ commutes with $S$ and
the rank of $T$ is 2, $T\in\mathcal{A}_{S}^{F}$ whence $\mathcal{A}_{S}$ is an
$L$-algebra. Being (SOT)-closed, $\mathcal{A}_{S}$ is a Lomonosov algebra.

Let us show that the space $\mathcal{X}^{F}:=\mathcal{A}_{S}^{F}%
\mathcal{X}:=\mathrm{lin}(\cup_{T\in\mathcal{A}_{S}^{F}}T\mathcal{X})$
coincides with $\mathcal{D}(S)$. Indeed, if $T\in\mathcal{A}_{S}^{F}$ then
$T\mathcal{D}(S)\subset\mathcal{D}(S)$ (this is true for all operators in
$\mathcal{A}_{S}$). Since $\mathcal{D}(S)$ is dense in $\mathcal{X}$, we have
that $T\mathcal{X}\subset\overline{T\mathcal{D}(S)}$. But $T\mathcal{D}(S)$ is
closed because $\dim T\mathcal{D}(S)<\infty$. Therefore $T\mathcal{X}\subset
T\mathcal{D}(S)\subset\mathcal{D}(S)$. It follows that $\mathcal{X}^{F}%
\subset\mathcal{D}(S)$. On the other hand, we have shown above that for
arbitrary $u\in\mathcal{D}(S)$, there is such $f\in\mathcal{D}(S^{\ast})$ that
the range of the operator $T_{u,f}\in\mathcal{A}_{S}^{F}$ contains $u$.
Therefore $\mathcal{D}(S)\subset\mathcal{X}^{F}$.

It follows that $S\in\mathbb{D}$, so $\mathcal{A}_{S}$ is not of real type.
Since the rank of $T_{u,f}$ equals 2, $\mathcal{A}_{S}$ is not of quaternion
type by Proposition \ref{Grank}. Thus $\mathcal{A}_{S}$ is of complex
type.\bigskip
\end{proof}

It will be convenient to formulate separately a statement which was obtained
during the proof of Theorem \ref{comType}.

\begin{proposition}
\label{domain} If $S$ is a closed complex partial structure then the subspace
$\mathcal{X}^{F}:=\mathcal{A}_{S}^{F}\mathcal{X}$ coincides with
$\mathcal{D}(S)$ and $S\in\mathbb{D}.$
\end{proposition}

Our next aim is to show that the algebras $\mathcal{A}_{S}$ majorize all
$L$-algebras of complex type.

\begin{theorem}
\label{ClasComp} Every L-algebra $\mathcal{A}$ of complex type is contained in
the algebra $\mathcal{A}_{S}$, for some closed partial complex structure $S$.
\end{theorem}

\begin{proof}
By our assumption, the commutant $\mathbb{D}$ of $\mathcal{A}^{F}$ on
$\mathcal{X}^{F}$ is isomorphic to $\mathbb{C}$. This means that
$\mathbb{D}=\mathrm{lin}\{\mathbf{1}_{\mathcal{X}^{F}},W\}$ where $W$ is an
operator on $\mathcal{X}^{F}$ commuting with all $K\in\mathcal{A}^{F}$ and
satisfying the condition $W^{-1}=-W$.

By Theorem \ref{closable}, $W$ is closable; let $S$ be its closure. By Lemma
\ref{closPCS}, $S$ is a PCS. We will consider the Lomonosov algebra
$\mathcal{A}_{S}$.

To show that $\mathcal{A}\subset\mathcal{A}_{S}$, we have to prove that
$TS\subset ST$, for each $T\in\mathcal{A}$. Firstly, let $x\in\mathcal{X}^{F}%
$. Then $Tx\in\mathcal{X}^{F}$ and $STx=TSx$ since $S|_{\mathcal{X}^{F}}=W$
and $T$ commutes with $W$ by Lemma \ref{more}.

Let now $x\in\mathcal{D}(S)$ and $y=Sx$. As $S$ is the closure of $W$, there
are $x_{n}\in\mathcal{X}^{F}$ with $x_{n}\rightarrow x$ and $Sx_{n}=W{x}%
_{n}\rightarrow y$. Therefore $Tx_{n}\rightarrow Tx$ and, by above,
$STx_{n}=TSx_{n}\rightarrow Ty$. It follows that $Tx\in\mathcal{D}(S)$ and
$STx=Ty=TSx$. Thus $TS\subset ST$ and we are done.\bigskip
\end{proof}

An operator algebra $\mathcal{A}$ is called \textit{local} (see \cite{FAM}) if
it coincides with the (SOT)-closure of $\mathcal{A}\cap\mathcal{K}%
(\mathcal{X})$, where $\mathcal{K}(\mathcal{X})$ is the ideal of all compact
operators. Now we will show that each algebra $\mathcal{A}_{S}$ has a property
which can be considered as a reinforced version of locality.

\begin{theorem}
\label{locNess} If $\mathcal{A}=\mathcal{A}_{S}$, where $S$ is a closed
partial complex structure, then $\mathcal{A}^{F}$ is \emph{(SOT)}-dense in
$\mathcal{A}$.
\end{theorem}

\begin{proof}
It suffices to show that for every finite-dimensional subspace $\mathcal{L}%
\subset\mathcal{X}$, and every $T\in\mathcal{A}$, the restriction
$T|_{\mathcal{L}}$ is close to $\mathcal{A}^{F}|_{\mathcal{L}}$. This means
that for any basis $\{e_{i}: i=1,...,\dim{\mathcal{L}}\}$ of $\mathcal{L}$ and
any $\varepsilon>0$, one can find $V\in\mathcal{A}^{F}$ with $\|Ve_{i}%
-Te_{i}\|<\varepsilon$.

Set $\mathcal{L}_{0} = \mathcal{L}\cap\mathcal{X}^{F}$. If $\mathcal{L}_{0}$
is not $\mathbb{C}$-linear (= $S$-invariant), then we denote by $\mathcal{L}%
^{\prime}_{0}$ its $\mathbb{C}$-linear span and replace $\mathcal{L}$ by
$\mathcal{L}+\mathcal{L}^{\prime}_{0}$. So we may assume that $\mathcal{L}%
_{0}$ is $\mathbb{C}$-linear. Denote by $C$ the norm of $S|_{\mathcal{L}_{0}}%
$. Since $S$ is antiinvolutive, $C\ge1$. Let $u_{1},...,u_{s}$ be a
$\mathbb{C}$-basis in $\mathcal{L}_{0}$, and let $w_{i} = Su_{i}$, $1\le i\le
s$. Let $\mathcal{M}$ be a complement of $\mathcal{L}_{0}$ in $\mathcal{L}$,
and $u_{s+1},...,u_{N}$ be a basis in $\mathcal{M}$. Then $\{u_{1},...,u_{N},
w_{1},...,w_{s}\}$ is a basis in $\mathcal{L}$.

Our main step will be the proof of the following\medskip

Claim: there is an operator $K\in\mathcal{A}^{F}$ such that the family
$Ku_{1},...,Ku_{N}$ is $\mathbb{C}$-linearly independent.\medskip

Indeed, if this is established then one can finish the proof as follows.

For all $i=1,...,N$ choose $q_{i}\in\mathcal{X}^{F}$ with $\Vert Tu_{i}%
-q_{i}\Vert<\varepsilon/C$. By strict $\mathbb{C}$-density of $\mathcal{A}%
^{F}$ on $\mathcal{X}^{F}$ (see Corollary \ref{Jac}), there is $R\in
\mathcal{A}^{F}$ with $RKu_{i}=q_{i}$, $i=1,...,N$. So, setting $V=RK,$ we get
all we need:
\[
\Vert Vu_{i}-Tu_{i}\Vert=\Vert q_{i}-p_{i}\Vert<\varepsilon/C\leq\varepsilon
\]
and, using the fact that $\mathcal{A}$ commutes with $S$ on $\mathcal{X}^{F}$
(see Lemma \ref{more}),
\[
\Vert Vw_{i}-Tw_{i}\Vert=\Vert VSu_{i}-TSu_{i}\Vert=\Vert S(V-T)u_{i}\Vert\leq
C\varepsilon/C=\varepsilon.
\]

We will prove the above Claim by induction. More precisely we will show by
induction that for each $n\le N$, there is an operator $K\in\mathcal{A}^{F}$
such that the family $Ku_{1},...,Ku_{n}$ is $\mathbb{C}$-linearly independent.

For $n=1$, the Claim is evident. Assume that it holds for $n=m<N$ and does not
hold for $n=m+1$. So for every $K\in\mathcal{A}^{F}$,
\[
Ku_{m+1}=\sum_{k=1}^{m}t_{k}(K)Ku_{k}%
\]
as in the proof of Lemma \ref{indep}, with the only difference that
coefficients $t_{k}(K)$ now belong to $\mathbb{C}$ (recall that the
multiplication by a complex number $a+ib$ acts in $\mathcal{X}^{F}$ as the
operator $a\mathbf{1_{\mathcal{X}^{F}}}+bS$). Repeating the arguments in the
proof of Lemma \ref{indep}), we again come to the conclusion that the numbers
$t_{k}(K)$ do not depend on $K$:
\[
t_{k}(K)=a_{k}\mathbf{1_{\mathcal{X}^{F}}}+b_{k}S.
\]

Thus
\begin{equation}
Ku_{m+1}=\sum_{k=1}^{m}(a_{k}+b_{k}S)Ku_{k},\text{ for all }K\in
\mathcal{A}^{F}. \label{eqM}%
\end{equation}
Let now $0\neq v\in\mathcal{X}^{F}=\mathcal{D}(S)$, $f\in\mathcal{D}(S^{\ast
})$ and $K=T_{v,f}$, the operator defined in (\ref{2ten}). Denoting $S^{\ast
}f$ by $g$, one can write
\[
Kx=f(x)v-g(x)Sv,\text{ for all }x\in\mathcal{X}.
\]
So we get from (\ref{eqM}) that
\[
f(u_{m+1})v-g(u_{m+1})Sv=\sum_{k=1}^{m}(a_{k}\mathbf{1}+b_{k}S)(f(u_{k}%
)v-g(u_{k})Sv),
\]
whence
\[
(f(u_{m+1})-\sum_{k=1}^{m}(a_{k}f(u_{k})+b_{k}g(u_{k})))v=(g(u_{m+1}%
)-\sum_{k=1}^{m}(a_{k}g(u_{k})-b_{k}f(u_{k})))Sv.
\]
Since $v$ and $Sv$ are not proportional, both parts are zero. Setting
$x_{0}=\sum_{k=1}^{m}a_{k}u_{k}$, $y_{0}=\sum_{k=1}^{m}b_{k}u_{k}$, we rewrite
this in the form:
\[
f(u_{m+1})=f(x_{0})+g(y_{0})\text{ and }g(u_{m+1})=g(x_{0})-f(y_{0}).
\]
Writing the second equality as $g(u_{m+1}-x_{0})=-f(y_{0})$ and setting
$z=x_{0}-u_{m+1},$ we have $(S^{\ast}f)(z)=f(y_{0})$ for all $f\in
\mathcal{D}(S^{\ast}).$ By (\ref{dualpair}), $z\in\mathcal{D}(S)$ and
$y_{0}=Sz=S(x_{0}-u_{m+1})$.

Since $\mathcal{D}(S)=\mathcal{X}^{F}$, by Proposition \ref{domain}, we see
that $u_{m+1}-x_{0}\in\mathcal{X}^{F}$. On the other hand, $x_{0}%
\in\mathrm{lin}(u_{1},...,u_{m})\subset\mathcal{L}$ and $u_{m+1}\in
\mathcal{L}$, so we obtain that $u_{m+1}-x_{0}\in\mathcal{X}^{F}%
\cap\mathcal{L}=\mathcal{L}_{0}$ whence $u_{m+1}\in\mathcal{L}_{0}%
+\mathrm{lin}(u_{1},...,u_{m})$. But this is possible only if $m+1\leq s$.

Indeed, let $q: \mathcal{L}\to\mathcal{L}/\mathcal{L}_{0}$ be the quotient
map. Then $q(u_{j}) = 0$, for $j\le s$. So if $m+1 > s$ then $q(u_{m+1})$ is a
linear combination of $(q(u_{j}))_{s<j<m+1}$. Since $\mathcal{M}%
\cap\mathcal{L}_{0} = \{0\}$ the map $q$ is injective on $\mathcal{M}$.
Therefore $u_{m+1}$ is a linear combination of $(u_{j})_{s<j<m+1}$, a contradiction.

Thus $m+1<s$, so we have that $y_{0}\in\mathcal{D}(S)$. Hence the validity of
the equality $f(u_{m+1})=f(x_{0})+g(y_{0})$ for all $f\in\mathcal{D}(S^{\ast
})$ implies that
\[
u_{m+1}=x_{0}+Sy_{0}=\sum_{k=1}^{m}(a_{k}u_{k}+b_{k}Su_{k})=\sum_{k=1}%
^{m}(a_{k}+ib_{k})u_{k},
\]
which contradicts the $\mathbb{C}$-linear independence of the family
$(u_{1},...,u_{s})$. The obtained contradiction proves the Claim. \bigskip
\end{proof}

It is not clear if closed PCSs exist in all real Banach spaces. If
$\mathcal{X}=\mathcal{H}$, an infinite-dimensional real Hilbert space, then
one can easily construct a bounded PCS --- it suffices to present
$\mathcal{H}$ as the orthogonal direct sum of two copies of a space
$\mathcal{H}_{0}$: $\mathcal{H}=\mathcal{H}_{0}\oplus\mathcal{H}_{0}$ and set
$S_{I}(x\oplus y)=(-y)\oplus x$. In fact we turn $\mathcal{H}$ into a complex
Hilbert space $\mathcal{H}_{c}$ if we denote the action of $S_{I}$ as the
multiplication by $i$.

Now one can introduce a large family of closed PCSs as follows. Choose in the
obtained complex Hilbert space $\mathcal{H}_{c}$ two closed $\mathbb{C}%
$-linear subspaces $\mathcal{M},\mathcal{N}$ forming a "generic pair":
\[
\mathcal{M}\cap\mathcal{N}=0,\;\;\;\overline{\mathcal{M}+{N}}=\mathcal{H}%
_{c}.
\]
Setting $\mathcal{D}=\mathcal{M}+{N}$ we define an operator $S_{\mathcal{M}%
,\mathcal{N}}$ on $\mathcal{D}$ by the rule
\[
S_{\mathcal{M},\mathcal{N}}(x+y)=ix-iy,\text{ for }x\in\mathcal{M}%
,y\in\mathcal{N};
\]
it is easy to check that $S_{\mathcal{M},\mathcal{N}}$ is a closed CPS in
$\mathcal{H}_{c}$. Indeed, if $x_{n}+y_{n}\rightarrow z$ and $ix_{n}%
-iy_{n}\rightarrow w$ where $x_{n}\in\mathcal{M}$, $y_{n}\in\mathcal{N}$,
then
\[
p:=(z-iw)/2=\lim x_{n}\in\mathcal{M}\text{ and }q:=(z+iw)/2=\lim y_{n}%
\in\mathcal{N}.
\]
So $z=p+q\in\mathcal{D}$, $w=ip-iq=S_{\mathcal{M},\mathcal{N}}%
(p+q)=S_{\mathcal{M},\mathcal{N}}z$ and $S_{\mathcal{M},\mathcal{N}}$ is
closed. The equality $S_{\mathcal{M},\mathcal{N}}^{-1}=-S_{\mathcal{M}%
,\mathcal{N}}$ is evident.

Using this construction we will show that at least in Hilbert spaces there are
many Lomonosov algebras of complex type.

\begin{corollary}
\label{cardinality} In a separable real infinite-dimensional Hilbert space
$\mathcal{H}$ there is a continuum of pairwise non-similar Lomonosov algebras
of complex type.
\end{corollary}

\begin{proof}
We will say that two linear subspaces $\mathcal{Y}_{1}, \mathcal{Y}_{2}$ of
Banach spaces $\mathcal{X}_{1}, \mathcal{X}_{2}$ are isomorphic, if there is a
bounded invertible operator $T: \mathcal{X}_{1} \to\mathcal{X}_{2}$ with
$T\mathcal{Y}_{1} = \mathcal{Y}_{2}$. Let us show that if partial complex
structures $S_{1}, S_{2}$ are such that algebras $\mathcal{A}_{S_{1}}$ and
$\mathcal{A}_{S_{2}}$ are similar then their domains $\mathcal{D}(S_{1}),
\mathcal{D}(S_{2})$ are isomorphic.

Indeed if $T\mathcal{A}_{S_{1}}T^{-1}=\mathcal{A}_{S_{2}}$, for some bounded
invertible operator $T$, then
\[
T(\mathcal{A}_{S_{1}}\cap\mathcal{F}(\mathcal{X}_{1}))T^{-1}=\mathcal{A}%
_{S_{2}}\cap\mathcal{F}(\mathcal{X}_{2}).
\]
Thus
\[
T(\mathcal{A}_{S_{1}}\cap\mathcal{F}(\mathcal{X}_{1}))\mathcal{X}%
_{1}=\mathcal{A}_{S_{2}}^{F}\mathcal{X}_{2}.
\]
By Proposition \ref{domain}, $T\mathcal{D}(S_{1})=\mathcal{D}(S_{2})$.

Now it remains to find in a real Hilbert space $\mathcal{H}$ a continuum of
closed PCS's whose domains are pairwise non-isomorphic.

Using the complex structure in $\mathcal{H}$ defined by the operator $S_{I}$,
we may assume that $\mathcal{H}$ is a complex space. The construction
described before the corollary, relates any generic pair $\mathcal{M}%
,\mathcal{N}$ of closed subspaces in $\mathcal{H}$ to the PCS
$S=S_{\mathcal{M},\mathcal{N}}$ with $\mathcal{D}(S)=\mathcal{M}+\mathcal{N}$.
So we will look for non-isomorphic subspaces of this form.

Recall that a linear subspace $\mathcal{L}$ of $\mathcal{H}$ is called an
\textit{operator range} if there is a Hilbert space $\mathcal{H}^{\prime}$ and
a bounded operator $W: \mathcal{H}^{\prime}\to\mathcal{H}$ with $W\mathcal{H}%
^{\prime}= \mathcal{L}$. It is easy to see that any domain of a closed
operator is an operator range.

Furthermore it is known (see \cite[Theorems 5.1, 5,11]{Dixm}, \cite[Theorem
2.6]{FW}) that any dense operator range $\mathcal{L}\varsubsetneq\mathcal{H}$
that contains an infinite-dimensional closed subspace of $\mathcal{H}$ can be
presented in the form $\mathcal{M}+\mathcal{N}$, for some generic pair
$\mathcal{M},\mathcal{N}$.

To any sequence $(H_{k})_{k=0}^{\infty}$ of pairwise orthogonal subspaces of
$\mathcal{H}$ such that $\mathcal{H}=\oplus_{k=0}^{\infty}H_{k}$ there
corresponds a dense linear subspace
\begin{equation}
\mathcal{L}_{(H_{k})}=\{\sum_{k=0}^{\infty}x_{k}:x_{k}\in H_{k},\sum
_{k=0}^{\infty}2^{k}\Vert x_{k}\Vert<\infty\}. \label{DimSum}%
\end{equation}
It is easy to show that $\mathcal{L}_{(H_{k})}$ is an operator range:
$\mathcal{L}_{(H_{k})}=T\mathcal{H}$, where $T=\sum_{k=0}^{\infty}%
2^{-k}P_{H_{k}}$. It was proved in \cite[Theorem 3.3]{FW} that if
$(K_{k})_{k=0}^{\infty}$ is another sequence of pairwise orthogonal subspaces,
then the subspaces $\mathcal{L}_{(H_{k})}$ and $\mathcal{L}_{(K_{k})}$ are
isomorphic if and only if the dimensions of $H_{k}$ and $K_{k}$ satisfy the
following condition:%

\begin{equation}
\text{ there is }p\in\mathbb{N}\text{ such that}\sum_{k=n}^{m}\dim H_{k}%
\leq\sum_{k=n-p}^{m+p}\dim K_{k},\text{ for any pair }(n,m)\text{ with }n<m,
\label{inneq}%
\end{equation}
and dually (it is assumed that $H_{k}=K_{k}=0$ if $k<0$).

Now for any $t>1$, we denote by $\mathcal{L}(t)$ a subspace of the form
$\mathcal{L}_{(H_{k})}$, where $\dim H_{k}=[k^{t}]:=\max\{n\in\mathbb{N}:n\leq
k^{t}\}$, for all $k>0$, while $\dim H_{0}=\infty$. The last condition
guarantees that $\mathcal{L}(t)$ contains an infinite-dimensional closed
subspace and therefore is a sum of two subspaces forming a generic pair. Let
us check that $\mathcal{L}(t)$ and $\mathcal{L}(r)$ are not isomorphic, if
$t\neq r$. Indeed, if $t>r$ and (\ref{inneq}) holds, for some $p$ and all
$m,n$, then choosing $n=p+1$ we obtain that
\begin{equation}
\sum_{k=p+1}^{m}[k^{t}]\leq\sum_{k=1}^{m+p}[k^{r}],\text{ for all }m.
\label{asymp}%
\end{equation}
However, since $a-1<[a]\leq a$, the left hand side of (\ref{asymp}) , when
$m\rightarrow\infty$, is asymptotically equivalent to
\[
\sum_{k=p+1}^{m}k^{t}\sim\int_{0}^{m}x^{t}dx=\frac{m^{t+1}}{t+1},
\]
while the right hand side, similarly, is asymptotically equivalent to
\[
\frac{(m+p)^{r+1}}{r+1}\sim\frac{m^{r+1}}{r+1}.
\]
So the inequality (\ref{inneq}) contradicts the condition $t>r$.
\end{proof}

\section{Lomonosov algebras of quaternion type}

If $\mathcal{A}$ is an $L$-algebra of quaternion type then, by definition,
there is an isomorphism $\pi_{\mathcal{A}}$ of the algebra $\mathbb{H}$ onto
$\mathbb{D}$, the commutant of $\mathcal{A}^{F}|_{\mathcal{X}^{F}}$. It can be
considered as a representation of $\mathbb{H}$ on the space $\mathcal{X}^{F}$.
If $G_{\mathbb{H}}\subset\mathbb{H}$ is the quaternion group
\[
G_{\mathbb{H}}=\{\pm1,\pm i.\pm j,\pm k:(-1)^{2}=1,\text{ }i^{2}=j^{2}%
=k^{2}=-1,\text{ }ijk=-1
\]
then the restriction of $\pi_{\mathcal{A}}$ to $G_{\mathbb{H}}$ will also be
denoted by $\pi_{\mathcal{A}}$.

It will be convenient to begin the study of this representation in a more
general context.

Let $G$ be a finite group with the unit element $e_{G}$, and let $\pi$ be a
representation of $G$ by linear operators on a dense linear subspace
$\mathcal{E}$ of a Banach space $\mathcal{X}$. The subspace $\mathcal{E}$ is
called \textit{the domain} of $\pi$. We will write $\mathcal{E}_{\pi}$ instead
of $\mathcal{E}$ when it will be necessary to underline that $\mathcal{E}$ is
the domain of $\pi$.

We say that $\pi$ is \textit{closed} if, whenever
\begin{equation}
\pi(g)x_{n}\rightarrow w(g)\in\mathcal{X}\text{ for all }g\in G\text{ and some
sequence }(x_{n})_{n=1}^{\infty}\text{ in }\mathcal{E}_{\pi}, \label{cloclo}%
\end{equation}
then $w(g)=\pi(g)x$, for some $x\in\mathcal{E}_{\pi}$ and all $g\in G$.

A representation $\pi$ is called \textit{closable} if, whenever a function $w:
G\to\mathcal{X}$ satisfies (\ref{cloclo}), the condition $w(e_{G})=0$ implies
$w(g)=0$, for all $g\in G.$

Let $C(G,\mathcal{X})$ be the Banach space of all $\mathcal{X}$-valued
functions on $G$ with the norm $\Vert F\Vert=\sup_{g\in G}\Vert F(g)\Vert$. We
denote by $\Gamma(\pi)$ the subset of the space $C(G,\mathcal{X})$ that
consists of all functions $F(g)=\pi(g)x$ where $x\in\mathcal{E}$.

\begin{lemma}
The following conditions are equivalent\emph{:\smallskip}

\emph{(i)} $\pi$ is closed\emph{;\smallskip}

\emph{(ii)} $\mathcal{E}$ is complete with respect to the norm $\Vert
x\Vert_{G}=\sum_{g\in G}\Vert\pi(g)x\Vert$\emph{;}\smallskip

\emph{(iii)} $\Gamma(\pi)$ is closed in $C(G,\mathcal{X})$.
\end{lemma}

\begin{proof}
(ii) $\Leftrightarrow$ (iii) follows from the fact that the map $x\mapsto
F(g)=\pi(g)x$ is a topological isomorphism between $(\mathcal{E},\|\cdot
\|_{G})$ and $\Gamma(\pi)$.

(i) $\Rightarrow$ (ii). Let $x_{n}$ be a Cauchy sequence in $(\mathcal{E}%
,\Vert\cdot\Vert_{G})$. Since $\Vert\pi(g)x\Vert\leq\Vert x\Vert_{G}$, the
sequence $\mathbf{\{}\pi(g)x_{n}\}$ , for any $g\in G$, is a Cauchy sequence
in $\mathcal{X}$. Hence, for each $g\in G$ there is $w(g)\in\mathcal{X}$ with
$\pi(g)x_{n}\rightarrow w(g)$. By (i), there is $x\in\mathcal{E}$ with
$w(g)=\pi(g)x$. Therefore $\Vert x_{n}-x\Vert_{G}\rightarrow0$.

(iii) $\Rightarrow$ (i). If $\pi(g)y_{n}\rightarrow w(g)$ for all $g$, then
the function $F(g)=w(g)$ belongs to the closure of $\Gamma(\pi)$. Since
$\Gamma(\pi)$ is closed, $F(g)$ belongs to $\Gamma(\pi)$. So there is
$y\in\mathcal{E}$ with $w(g)=\pi(g)y$.\bigskip
\end{proof}

Let $\overline{\Gamma}$ be the closure of $\Gamma(\pi)$ in $C(G,\mathcal{X})$.
It is easy to see that $\Gamma(\pi)$ and $\overline{\Gamma}$ are linear
subspaces of $C(G,\mathcal{X})$ and
\begin{equation}
F\in\overline{\Gamma}\text{ if there are }x_{n}\in\mathcal{E}_{\pi}\text{ such
that }F(g)=\underset{n\rightarrow\infty}{\lim}\pi(g)x_{n}\text{ for all }g\in
G. \label{3.5}%
\end{equation}
Setting $\mathcal{Z}=\{F(e_{G}):F\in\overline{\Gamma}\},$ we see that
$\mathcal{Z}$ is a linear subspace of $\mathcal{X}$ containing $\mathcal{E}%
_{\pi}$.

\begin{lemma}
\label{closure}If $\pi$ is closable then $\overline{\Gamma}=\Gamma(\rho)$
where $\rho$ is a closed representation of $G$ on $\mathcal{Z}$.
\end{lemma}

\begin{proof}
Firstly, we claim that, for each $g\in G$, there is a linear operator
$\rho(g)$: $\mathcal{Z}\rightarrow\mathcal{X}$ with
\begin{equation}
F(g)=\rho(g)F(e_{G})\text{ for each function }F\in\overline{\Gamma}.
\label{3.7}%
\end{equation}
Indeed, let $F\in\overline{\Gamma}.$ By (\ref{3.5}), there is a sequence
$\{x_{n}\}$ in $\mathcal{E}_{\pi}$ with $F(g)=\lim_{n}\pi(g)x_{n}$. So the
assumption of closability of $\pi$ implies that if $F(e_{G})=0$ then $F(g)=0$,
for all $g\in G$. Since the maps $F\mapsto F(e_{G})\in\mathcal{Z}$ and
$F\mapsto F(g)$ are linear, this shows that $F(g)$ depends on $F(e_{G})$
linearly and our claim is proved.

Now we have to show that $\rho(g)\mathcal{Z}\subset\mathcal{Z}$ and
$\rho(gh)=\rho(g)\rho(h)$, for all $g,h\in G$.

Let $x\in\mathcal{Z}$. Then there is $F\in\overline{\Gamma}$ with
$F(e_{G})=x.$ So, by (\ref{3.5}), there are $x_{n}\in\mathcal{E}_{\pi}$ such
that
\begin{equation}
\pi(g)x_{n}\rightarrow F(g)=\rho(g)F(e_{G})=\rho(g)x\text{ for all }g\in G.
\label{3.4}%
\end{equation}
Let $t_{n}=\pi(h)x_{n}$ for some $h\in G$. Set $F_{h}(g)=F(gh).$ Then
$t_{n}\in\mathcal{E}_{\pi}$ and, by (\ref{3.4}),
\begin{equation}
t_{n}\rightarrow\rho(h)x\text{ and }\pi(g)t_{n}=\pi(gh)x_{n}\rightarrow
\rho(gh)x=F(gh)=F_{h}(g). \label{3.6}%
\end{equation}
It follows from (\ref{3.5}) and (\ref{3.6}) that $F_{h}\in\overline{\Gamma}.$
Hence $\rho(h)x=F(h)=F_{h}(e_{G})\in\mathcal{Z}$ and%
\[
\rho(gh)x=F(gh)=F_{h}(g)\overset{(\ref{3.7})}{=}\rho(g)F_{h}(e_{G}%
)=\rho(g)\rho(h)x
\]
Since also $\rho(e_{G})x=x$, we have that $g\mapsto\rho(g)$ is a
representation of $G$ on $\mathcal{Z}$.

Since $\Gamma(\rho)=\overline{\Gamma(\pi)}$, the representation $\rho$ is
closed.\bigskip
\end{proof}

The representation $\rho$ defined in the proof of Lemma \ref{closure} is
called \textit{the closure of} $\pi$ and denoted by $\overline{\pi}$.

For a representation $\pi$ of $G$ on $\mathcal{E}$, we set
\begin{equation}
\mathcal{E}^{\ast}=\{f\in\mathcal{X}^{\ast}\text{: }|f(\pi(g)x)|\leq C\Vert
x\Vert,\text{ for some }C>0\text{ and all }x\in\mathcal{E},g\in G\}.
\label{dualspace}%
\end{equation}
We call a representation $\pi$ \textit{regular} if the space $\mathcal{E}%
^{\ast}$ is weak*-dense in $\mathcal{X}^{\ast}$.

\begin{lemma}
\label{RegClos}\emph{(i) }Every regular representation is closable.\smallskip

\emph{(ii)} The closure of a regular representation is regular.
\end{lemma}

\begin{proof}
Let $\pi$ be a regular representation of a group $G$.

(i) If $(x_{n})_{n=1}^{\infty}\in\mathcal{E}_{\pi}$, $\lim\pi(g)x_{n}=w(g)$
and $w(e_{G})=0$, then $\Vert x_{n}\Vert\rightarrow0$ and therefore
$|f(\pi(g)x_{n})|\leq C\Vert x_{n}\Vert\rightarrow0$, for any $g\in G$ and any
$f\in\mathcal{E}^{\ast}$. It follows that $f(w(g))=\lim f(\pi(g)x_{n})=0.$
Hence $w(g)=0,$ since by our assumptions $\mathcal{E}^{\ast}$ is weak*-dense
in $\mathcal{X}^{\ast}$. We proved (i).

(ii) Let $\mathcal{E}_{\overline{\pi}}$ be the domain of $\overline{\pi}$. To
prove (ii) it suffices to show that $\mathcal{E}^{\ast}\subset\mathcal{E}%
_{\overline{\pi}}^{\ast}$, that is, $|f(\overline{\pi}(g)x)|\leq C\Vert
x\Vert$, for all $g\in G$, $f\in\mathcal{E}_{\overline{\pi}}^{\ast}$ and
$x\in\mathcal{E}_{\overline{\pi}}$.

Let $x\in\mathcal{E}_{\overline{\pi}}$. By (\ref{3.4}), there are $x_{n}%
\in\mathcal{E}_{\pi}$ with $x_{n}\rightarrow x$ and $\pi(g)x_{n}%
\rightarrow\overline{\pi}(g)x$. Therefore $f(\overline{\pi}(g)x)=\lim
f(\pi(g)x_{n})$, so that $|f(\overline{\pi}(g)x)|\leq\lim|f(\pi(g)x_{n}%
)|\leq\limsup C\Vert x_{n}\Vert=C\Vert x\Vert$.\bigskip
\end{proof}

For each $g\in G$ and each $f\in\mathcal{E}^{\ast}$, the map $x\mapsto
f(\pi(g^{-1})x)$ is a bounded linear functional on $\mathcal{E}_{\pi}$. The
extension of this functional to $\mathcal{X}$ by continuity will be denoted by
$\pi^{\ast}(g)f$. It is easy to check that in this way we define a
representation of $G$ on $\mathcal{E}^{\ast}$; we will denote this
representation by $\pi^{\ast}$.

For every operator $K$ acting on $\mathcal{E}_{\pi}$, its "mean"\ $M_{G}(K)$
is defined by the formula
\begin{equation}
M_{G}(K)=\sum_{g\in G}\pi(g)K\pi(g^{-1}). \label{mean}%
\end{equation}
It is easy to check that $M_{G}(K)$ commutes with all operators $\pi(g)$.

In particular, for $y\in\mathcal{E}_{\pi}$ and $f\in\mathcal{E}^{\ast}$, we
set
\begin{equation}
T_{y,f}=M_{G}(y\otimes f)=\sum_{g\in G}\pi(g)(y\otimes f)\pi(g^{-1}%
)=\sum_{g\in G}\pi(g)y\otimes\pi^{\ast}(g)f, \label{mean1}%
\end{equation}
where as usual $y\otimes f$ is the rank one operator on $\mathcal{E}$ acting
by the rule
\[
(y\otimes f)(x)=f(x)y.
\]

Since the operators $\pi(g)y\otimes\pi^{\ast}(g)f$ are in fact defined and
bounded on $\mathcal{X}$ we may (and will) consider $T_{y,f}$ as an operator
on $\mathcal{X}$. Clearly $\mathrm{rank}(T_{y,f})\le|G|$ so
\[
T_{y,f}\in\mathcal{F}(\mathcal{X}).
\]

If $\pi$ is a representation of a group $G$ on a dense subspace $\mathcal{E}$
of a real Banach space $\mathcal{X}$ then we denote by $\mathcal{A}_{\pi}$ the
set of all operators $T\in B(\mathcal{X})$ that preserve $\mathcal{E}$ and
commute with all operators $\pi(g)$ on $\mathcal{E}$:
\[
\mathcal{A}_{\pi}=\{T\in B(\mathcal{X})\text{: }T\mathcal{E\subseteq E}\text{
and }T\pi(g)|_{\mathcal{E}}=\pi(g)T|_{\mathcal{E}}\text{ for }g\in G\}.
\]

\begin{proposition}
\label{predst} If a representation $\pi$ is closed then $\mathcal{A}_{\pi}$ is
a \emph{(SOT)}-closed algebra of operators on $\mathcal{X}$.
\end{proposition}

\begin{proof}
Clearly, $\mathcal{A}_{\pi}$ is an algebra. If $T_{n}\overset{sot}{\rightarrow
}T$, for some $T_{n}\in\mathcal{A}_{\pi}$ and $T\in\mathcal{B}(\mathcal{X})$,
then%
\[
\pi(g)T_{n}x=T_{n}\pi(g)x\rightarrow T\pi(g)x\text{ for any }x\in
\mathcal{E},g\in G.
\]
Since $\pi$ is closed on $\mathcal{E}$ and $T_{n}x\rightarrow Tx$,
$T\pi(g)x=\pi(g)y$ for some $y\in\mathcal{E}$. In particular, $Tx=y\in
\mathcal{E}$, so that $T$ preserves $\mathcal{E}$, and $T\pi(g)x=\pi(g)Tx$.
Thus $T$ commutes with $\pi(G)$ on $\mathcal{E}$. It follows that
$T\in\mathcal{A}_{\pi}$ and $\mathcal{A}_{\pi}$ is SOT-closed.\bigskip
\end{proof}

In what follows we consider only the quaternion group $G=G_{\mathbb{H}}$ and
those representations of $G_{\mathbb{H}}$ on a linear subspace $\mathcal{E}$
of $\mathcal{X}$ that satisfy the condition
\[
\pi(1)=-\pi(-1)=\mathbf{1}_{\mathcal{E}}.
\]
Any such representation extends to a representation of the algebra
$\mathbb{H}$: for $q=\alpha+\beta i+\gamma j+\delta k\in\mathbb{H}$, one sets%
\[
\pi(q)=\alpha\mathbf{1}_{\mathcal{E}}+\beta\pi(i)+\gamma\pi(j)+\delta\pi(k).
\]

One can obtain examples of closed representations of $G_{\mathbb{H}}$ as
follows. Let $\mathcal{X}=l^{2}$ with the standard basis $\{e_{n}%
\}_{n\in\mathbb{N}}$ and let $\mathcal{X}_{m}=span\{e_{4m-3},e_{4m-2}%
,e_{4m-1},e_{4m}\}$, $m=1,2,...$. We identify $\mathcal{X}$ with the
orthogonal direct sum of all $\mathcal{X}_{m}$. For each $m$, we choose a
linear bijection $\phi_{m}$: $\mathbb{H}\rightarrow\mathcal{X}_{m}$ and define
a representation $\pi_{m}$ of $\mathbb{H}$ on $\mathcal{X}_{m}$ by
\[
\pi_{m}(q)x=\phi_{m}(q\phi_{m}^{-1}(x))\text{ for } q\in\mathbb{H}\text{ and
}x\in\mathcal{X}_{m}.
\]

Let $s_{m}=\max_{g\in G_{\mathbb{H}}}\Vert\pi_{m}(g)\Vert$, and let
$\mathcal{E}$ be the set of all elements $x=\oplus_{m=1}^{\infty}x_{m}%
\in\mathcal{X}$ for which $\sum_{m}s_{m}^{2}\Vert x_{m}\Vert^{2}<\infty$. The
representation $\pi$ of $G_{\mathbb{H}}$ on $\mathcal{E}$ is defined as the
direct sum of representations $\pi_{m}$.

\begin{theorem}
\label{Api}If a representation $\pi$ of the group $G=G_{\mathbb{H}}$ on a
dense\textbf{ }subspace $\mathcal{E}\subset\mathcal{X}$ is closed and
regular\emph{,} then $\mathcal{A}_{\pi}$ is a Lomonosov algebra of quaternion
type on $\mathcal{X}$.
\end{theorem}

\begin{proof}
By Proposition \ref{predst}, $\mathcal{A}_{\pi}$ is a (SOT)-closed algebra.
Note that all operators $T_{y,f}$ defined by the formula (\ref{mean1}) for any
pair $(y,f),$ where $y\in\mathcal{E}$, $f\in\mathcal{E}^{\ast}$, belong to
$\mathcal{A}_{\pi}$, since they preserve $\mathcal{E}$ (moreover, they map
$\mathcal{X}$ to $\mathcal{E}$) and commute with $\pi(G)$ on $\mathcal{E}$.
Let $0\neq y\in\mathcal{E}$ and $0\neq x\in\mathcal{X}$; we are going to find
$f\in\mathcal{E}^{\ast}$ such that $T_{y,f}x=y$.

We claim that there exists $f\in\mathcal{E}^{\ast}$ satisfying conditions
\begin{equation}
f(x)=\frac{1}{2}\text{ and }(\pi^{\ast}(i)f)(x)=(\pi^{\ast}(j)f)(x)=(\pi
^{\ast}(k)f)(x)=0. \label{popolam}%
\end{equation}
To prove the claim, we consider the linear maps $f\mapsto f(x)$ and $v$:
$f\mapsto v(f)$ on $\mathcal{E}^{\ast}$, where
\[
v(f)=((\pi^{\ast}(i)f)(x),(\pi^{\ast}(j)f)(x),(\pi^{\ast}(k)f)(x))\in
\mathbb{R}^{3}.
\]
It suffices to show that there is $f\in\mathcal{E}^{\ast}$ with $v(f)=0$ while
$f(x)\neq0$. If, by contradiction, $v(f)=0$ implies $f(x)=0,$ then the map
$\nu(f)\mapsto f(x),$ $f\in\mathcal{E}^{\ast},$ is well defined and linear on
a subspace $\nu(\mathcal{E}^{\ast})\subseteq\mathbb{R}^{3}.$ Extending it to a
linear functional on $\mathbb{R}^{3}$ we get that there exist $\alpha
,\beta,\gamma\in\mathbb{R}$ with
\[
f(x)=\alpha(\pi^{\ast}(i)f)(x)+\beta(\pi^{\ast}(j)f)(x)+\gamma(\pi^{\ast
}(k)f)(x),\text{ for all }f\in\mathcal{E}^{\ast}.
\]
This can be written in the form
\[
(\pi^{\ast}(q)f)(x)=0,\text{ for all }f\in\mathcal{E}^{\ast},\text{ where
}q=1-\alpha i-\beta j-\gamma k.
\]
In other words, $\pi^{\ast}(q)(\mathcal{E}^{\ast})\subset\mathcal{E}_{0}%
^{\ast}$ where $\mathcal{E}_{0}^{\ast}=\{f\in\mathcal{E}^{\ast}:f(x)=0\}$.
Since $q$ is invertible in $\mathbb{H}$, the operator $\pi^{\ast}(q)$ is
invertible. So we get the equality $\mathcal{E}^{\ast}=\mathcal{E}_{0}^{\ast}%
$. This contradicts the assumption that $\pi$ is regular, because the subspace
$\mathcal{E}_{0}^{\ast}$ is not weak$^{\ast}$-dense in $\mathcal{X}^{\ast}$.

Let now $f\in\mathcal{E}^{\ast}$ satisfy (\ref{popolam}). Then, by
(\ref{mean1}),
\begin{align*}
T_{y,f}x  &  =\sum_{g\in G}(\pi(g)^{-1}y\otimes\pi(g)^{\ast}f)x=\sum_{g\in
G}(\pi(g)^{\ast}f)(x)\pi(g)^{-1}y\\
&  =(\pi^{\ast}(1)f)(x)\pi(1)y+(\pi^{\ast}(-1)f)(x)\pi(-1)y=2f(x)y=y.
\end{align*}
Thus we proved that, for any $0\neq x\in\mathcal{X}$, the space $\mathcal{A}%
_{\pi}x$ contains $\mathcal{E}$. Since $\overline{\mathcal{E}}=\mathcal{X}$,
$\mathcal{A}_{\pi}$ is transitive on $\mathcal{X}$.

Taking in account that all operators $T_{y,f}$ are of finite rank, we conclude
that $\mathcal{A}_{\pi}$ is an $L$-algebra. Since $\mathcal{A}_{\pi}$ is
(SOT)-closed, it is a Lomonosov algebra. As $\dim\pi(\mathbb{H})=4$ (the
representation $\pi$ of $\mathbb{H}$ is injective because $\mathbb{H}$ has no
non-trivial ideals) and the commutant $\mathbb{D}$ of $\mathcal{A}_{\pi}$
contains $\pi(\mathbb{H})$, we conclude that $\mathbb{D}=\pi(\mathbb{H})$ and
$\mathcal{A}_{\pi}$ is a Lomonosov algebra of quaternion type.
\end{proof}

\begin{corollary}
\label{domain2}In assumptions of Theorem \emph{\ref{Api}} the space
$\mathcal{X}^{F}:=\mathcal{A}_{\pi}^{F}\mathcal{X}$ coincides with
$\mathcal{E}$.
\end{corollary}

\begin{proof}
Let $x\in\mathcal{E}$. By Theorem \ref{Api}, there is $f\in\mathcal{E}^{\ast}$
satisfying condition (\ref{popolam}). It follows that $T_{x,f}x=x,$ where
$T_{x,f}\in\mathcal{A}_{\pi}^{F}$ is the operator defined in (\ref{mean1})
with $y=x$. Therefore $x\in\mathcal{X}^{F}$ whence $\mathcal{E}\subset
\mathcal{X}^{F}$. On the other hand, if $T\in\mathcal{A}_{\pi}^{F}$ then
$T\mathcal{X}=T\overline{\mathcal{E}}\subset\overline{T\mathcal{E}}$. Since
$T\mathcal{E}$ is finite-dimensional, $\overline{T\mathcal{E}}=T\mathcal{E}%
\subset\mathcal{E}$. Therefore $\mathcal{X}^{F}\subset\mathcal{E}$.\bigskip
\end{proof}

Now we return to the representations $\pi_{\mathcal{A}}$ of $G_{\mathbb{H}}$
related to arbitrary $L$-algebras $\mathcal{A}$ of quaternion type (see the
discussion at the beginning of Section 4). Recall that operators
$\pi_{\mathcal{A}}(g)$ act on the space $\mathcal{X}^{F}=\mathcal{A}%
^{F}\mathcal{X}$, so $\mathcal{E}_{\pi_{\mathcal{A}}} = \mathcal{X}^{F}.$

\begin{lemma}
\label{regularity} Every representation $\pi_{\mathcal{A}}$ is regular.
\end{lemma}

\begin{proof}
For any functional $f\in\mathcal{X}^{\ast}$ and operator $T\in\mathcal{A}^{F}%
$, we denote by $f_{T}$ the functional on $\mathcal{X}$ defined by the
equality $f_{T}(x)=f(Tx)$. Then, for each $x\in\mathcal{X}^{F}$,
\[
|f_{T}(\pi_{\mathcal{A}}(g)x)|=|f_{T}(T\pi_{\mathcal{A}}(g)x)|=|f(\pi
_{\mathcal{A}}(g)Tx)|\leq\Vert f\Vert\Vert\pi_{\mathcal{A}}(g)\Vert_{T}\Vert
T\Vert\Vert x\Vert,
\]
where $\Vert\pi_{\mathcal{A}}(g)\Vert_{T}$ is the norm of the restriction of
$\pi_{\mathcal{A}}(g)$ to the finite-dimensional subspace $T\mathcal{X}$.
Since $G_{\mathbb{H}}$ is finite, $\sup_{g\in G_{\mathbb{H}}}\Vert
\pi_{\mathcal{A}}(g)\Vert_{T}<\infty$ whence
\[
|f_{T}(\pi(g)x)|<C\Vert x\Vert,\text{ for all }g\in G_{\mathbb{H}} \text{ AND
} x\in\mathcal{X}^{F}.
\]
Hence all functionals $f_{T}$ belong to $(\mathcal{X}^{F})^{\ast}$ (see
(\ref{dualspace})). So to see that the subspace $(\mathcal{X}^{F})^{\ast}$ is
weak*-dense in $\mathcal{X}^{\ast}$ it suffices to show that the intersection
of kernels of all functionals $f_{T}$ is $\{0\}$. If $x\in\mathcal{X}$ is such
that $f_{T}(x)=0$, for all $f\in\mathcal{X}^{\ast}$ and all $T\in
\mathcal{A}^{F},$ then $Tx\in\cap_{f\in\mathcal{X}^{\ast}}\ker f=\{0\}$, for
all $T\in\mathcal{A}^{F}$. Hence $x\in\ker\mathcal{A}^{F}=\{0\}$.
\end{proof}

\begin{theorem}
\label{ClasQuat} Every $L$-algebra $\mathcal{A}$ of quaternion type is
contained in the Lomonosov algebra $\mathcal{A}_{\pi}$ of quaternion
type\emph{,} where $\pi$ is some closed regular representation of
$G_{\mathbb{H}}$.
\end{theorem}

\begin{proof}
It follows from Lemma \ref{regularity} that the representation $\pi
_{\mathcal{A}}$ of $G_{\mathbb{H}}$ on $\mathcal{X}^{F}$ is regular. Hence, by
Lemma \ref{RegClos}, it is closable and its closure $\pi=\overline
{\pi_{\mathcal{A}}}$ is a regular closed representation. So by Theorem
\ref{Api}, the algebra $\mathcal{A}_{\pi}$ is a Lomonosov algebra of
quaternion type.

The domain $\mathcal{E}_{\pi}$ of $\pi$ clearly contains the subspace
$\mathcal{E}_{\pi_{\mathcal{A}}}=\mathcal{X}^{F}$ of $\mathcal{X}$. For
$z\in\mathcal{E}_{\pi}$, there are $y_{n}\in\mathcal{X}^{F}$ such that
$\pi_{\mathcal{A}}(g)y_{n}\rightarrow\pi(g)z$, for all $g\in G$. Hence, for
each $T\in\mathcal{A}$, the sequence $t_{n}=Ty_{n}$ satisfies conditions
\[
t_{n}\rightarrow Tz\text{ and }\pi_{\mathcal{A}}(g)t_{n}=\pi_{\mathcal{A}%
}(g)Ty_{n}=T\pi_{\mathcal{A}}(g)y_{n}\rightarrow T\pi(g)z.
\]
It follows that $Tz\in\mathcal{E}_{\pi}$ and $T\pi(g)z=\pi(g)Tz$.

Thus all operators in $\mathcal{A}$ preserve the subspace $\mathcal{E}_{\pi}$
and commute with operators $\pi(g)$ on $\mathcal{E}_{\pi}$. By the definition
of $\mathcal{A}_{\pi}$, this means that $\mathcal{A}\subset\mathcal{A}_{\pi}%
$.\bigskip
\end{proof}

We will need a special construction of closed regular representations of the
group $G_{\mathbb{H}}$, resembling the construction of Lomonosov algebras of
complex type considered in the previous section.

Let ${\mathcal{H}}_{q}$ be an infinite-dimensional separable\textbf{
}quaternion Hilbert space. The multiplication by $i,j,k$ defines a
representation of $G_{\mathbb{H}}$ in ${\mathcal{H}}_{q}$ which we denote by
$\tau$. Let $\mathcal{M},\mathcal{N}$ be a generic pair of quaternion-linear
subspaces in ${\mathcal{H}}_{q}$:
\[
\mathcal{M}\cap\mathcal{N}=\{0\}\text{ and }\overline{\mathcal{M}+\mathcal{N}%
}={\mathcal{H}}_{q}.
\]
On the space $\mathcal{E}=\mathcal{M}+\mathcal{N}$ we define a representation
$\pi$ of $G_{\mathbb{H}}$ by setting
\begin{equation}
\pi(g)(x+y)=\tau(g)x+\tau(\alpha(g))y,\text{ for }x\in\mathcal{M}%
,y\in\mathcal{N}, \label{4.1}%
\end{equation}
where $\alpha$ is an automorphism of $G_{\mathbb{H}}$ such that $\alpha(i)=-i$.

\begin{lemma}
\label{CloReg} The representation $\pi$ is closed and regular.
\end{lemma}

\begin{proof}
Let $\pi(g)(x_{n}+y_{n})\rightarrow z(g)$ for all $g\in G_{\mathbb{H}}$.
Taking subsequently $g=1$ and $g=i$, we obtain that
\[
x_{n}+y_{n}\rightarrow z(1),\text{ and }ix_{n}-iy_{n}\rightarrow z(i),
\]
whence
\[
x_{n}\rightarrow x:=\frac{1}{2}(z(1)-iz(i))\in\mathcal{M}\text{ and }%
y_{n}\rightarrow y:=\frac{1}{2}(z(1)+iz(i))\in\mathcal{N}.
\]
It follows that, for all $g\in G_{\mathbb{H}}$,
\[
\pi(g)(x_{n}+y_{n})=\tau(g)x_{n}+\tau(\alpha(g))y_{n}\rightarrow\tau
(g)x+\tau(\alpha(g))y=\pi(g)(x+y).
\]
Thus $z(g)=\pi(g)(x+y)$, so that the representation $\pi$ is closed.

Furthermore, let $f\in\mathcal{N}^{\bot}$, the annihilator of $\mathcal{N}$ in
$\mathcal{X}^{\ast}$. Since $\mathcal{N}$ is $\tau$-invariant, we get that,
for each $x\in\mathcal{M},y\in\mathcal{N}$ and $g\in G$,
\[
f(\pi(g)(x+y))=f(\tau(g)x+\tau(\alpha(g))y)=f(\tau(g)x)=f(\tau(g)x+\tau(g)y),
\]
whence $|f(\pi(g)(x+y))|\leq\Vert f\Vert\Vert x+y\Vert$. Thus $f\in
\mathcal{E}^{\ast}$, so that $\mathcal{N}^{\bot}\subset\mathcal{E}^{\ast}$.
Similarly $\mathcal{M}^{\bot}\subset\mathcal{E}^{\ast}$. Since $(\mathcal{M}%
^{\bot}+\mathcal{N}^{\bot})_{\bot}=\mathcal{M}\cap\mathcal{N}=\{0\}$, we
obtain that $\mathcal{E}^{\ast}$ is weak*-dense in $\mathcal{X}^{\ast}$. Hence
the representation $\pi$ is regular.\bigskip
\end{proof}

We denote the representation $\pi$ defined in (\ref{4.1}) by $\pi
_{\mathcal{M},\mathcal{N}}$; the corresponding Lomonosov algebra
$\mathcal{A}_{\pi}$ is denoted by $\mathcal{A}_{\mathcal{M},\mathcal{N}}$.

\begin{corollary}
\label{continuum}In an infinite-dimensional\emph{,} separable real Hilbert
space there is a continuum of pairwise non-similar Lomonosov algebras of
quaternion type.
\end{corollary}

\begin{proof}
An infinite-dimensional real Hilbert space $\mathcal{H}$ can be realized as a
tensor product of a four-dimensional real space $l_{4}^{2}$ and a real Hilbert
space $\mathcal{H}_{0}$. Identifying in a natural way $l_{4}^{2}$ with
$\mathbb{H}$, one can turn $\mathcal{H}=\mathbb{H}\otimes\mathcal{H}_{0}$ into
a quaternion Hilbert space by setting
\[
q_{1}(q_{2}\otimes x)=q_{1}q_{2}\otimes x,
\]
for $q_{1},q_{2}\in$, $x\in\mathcal{H}_{0}$. Let $\mathcal{M}_{0}%
,\mathcal{N}_{0}$ be a generic pair of subspaces in ${\mathcal{H}}_{0}.$ Then
the quaternion linear subspaces $\mathcal{M}=\mathbb{H}\otimes\mathcal{M}_{0}%
$, $\mathcal{N}=\mathbb{H}\otimes\mathcal{N}_{0}$ form a generic pair in
${\mathcal{H}}$.

We have to show that there is a continuum of pairwise non-similar algebras of
the form $\mathcal{A}_{\mathcal{M},\mathcal{N}}$. Arguing as in the proof of
Corollary \ref{cardinality} and using Corollary \ref{domain2} instead of
Corollary \ref{domain}, we see that it suffices to show that there is a
continuum of non-isomorphic subspaces $\mathcal{E}=\mathcal{M}+\mathcal{N}$ of
the form $\mathcal{M}=\mathbb{H}\otimes\mathcal{M}_{0}$, $\mathcal{N}%
=\mathbb{H}\otimes\mathcal{N}_{0}$, for some generic pairs $(\mathcal{M}%
_{0},\mathcal{N}_{0})$.

We will choose subspaces $\mathcal{E}$ in the form $\mathcal{L}_{(H_{k})}$
(see (\ref{DimSum})). Namely, for each $t>1$, we take $H_{k}(t) =
\mathbb{H}\otimes H_{k}^{0}(t)$, where $\dim H_{k}^{0}(t) = [k^{t}]$ for
$k>0$, and $\dim H_{0}^{0} = \infty$. Then $\dim_{\mathbb{R}} H_{k}(t) =
4[k^{t}]$ and it follows from (\ref{inneq}) that the subspaces $\mathcal{L}%
_{(H_{k}(t))}$ are pairwise non-isomorphic. On the other hand, since the
subspaces $\mathcal{L}_{(H_{k}^{0}(t))}$ are operator ranges containing
infinite-dimensional closed subspaces, they can be presented in the form
$\mathcal{M}_{0}(t)+\mathcal{N}_{0}(t)$, for generic pairs $\mathcal{M}%
_{0}(t),\mathcal{N}_{0}(t)$. Therefore
\[
\mathcal{L}_{(H_{k}(t))} =\mathcal{M}(t)+\mathcal{N}(t) = \mathbb{H}%
\otimes\mathcal{M}_{0}(t) + \mathbb{H}\otimes\mathcal{N}_{0}(t)
\]
and we are done.
\end{proof}

\section{ Questions and commentaries}

1. Is it true that $\mathcal{A}=\overline{\mathcal{A}^{F}}$, for every
Lomonosov algebra $\mathcal{A}$? (As always the bar over a set of operators
denotes the closure in SOT).\medskip

2. Is it true that each Lomonosov algebra contains the identity operator?

\begin{proposition}
\label{eqv}The questions \emph{1} and \emph{2} are equivalent.
\end{proposition}

\begin{proof}
Firstly, we show that if the answer to Question 2 is positive then the answer
to Question 1 is positive. Indeed, under this assumption, for any Lomonosov
algebra $\mathcal{A}$, the closure $\overline{\mathcal{A}^{F}}$ of
$\mathcal{A}^{F}$ is unital. Since $\overline{\mathcal{A}^{F}}$ is an ideal of
$\mathcal{A}$, it coincides with $\mathcal{A}$.

Conversely, assume that the answer to Question 1 is positive. If a Lomonosov
algebra $\mathcal{A}$ is not unital, then $\mathcal{B}=\mathcal{A}%
+\mathbb{R}\mathbf{1}$ is a Lomonosov algebra, so that $\mathcal{B}%
=\overline{\mathcal{B}^{F}}$ by our assumption. Therefore $\mathcal{B}^{F}$ is
not contained in $\mathcal{A}$, so there is a finite rank operator
$K\in\mathcal{B}^{F}\setminus\mathcal{A}$. Clearly, $K=\lambda\mathbf{1}+T$
where $T\in\mathcal{A}$ and $0\neq\lambda\in\mathbb{R}.$ Thus $T=-\lambda
\mathbf{1}+K\in\mathcal{A}$ whence the algebra $\mathcal{\mathcal{C}%
\subset\mathcal{A}}$ generated by $T$ contains an operator $T_{1}$ of the form
$\mathbf{1}+R$ where $R\in\mathcal{F}(\mathcal{X})$. Since $\mathcal{C}$ is
commutative and finite-dimensional, idempotents lift from any quotient of
$\mathcal{C}$ (much more general results can be found e.g. in \cite{Nic}).
This means that if $\mathcal{J}$ is an ideal of $\mathcal{C}$ and $W^{2}-W
\in\mathcal{J}$, for some $W\in\mathcal{C}$, then there is $P\in\mathcal{C}$
such that $P^{2}=P$ and $P-W \in\mathcal{J}$. Applying this to $\mathcal{J}%
=\mathcal{C}\cap\mathcal{F}(\mathcal{X})$ and $W = T_{1}$ we get that there is
a projection $P\in\mathcal{A}$ such that $P-1-R\in\mathcal{F}(\mathcal{X})$.
Thus $\mathcal{A}$ contains a projection $P$ such that the projection $Q=1-P$
is of finite rank.

For $Z\in\mathcal{A}$, the operators $QZQ=(\mathbf{1}-P)Z(\mathbf{1}%
-P)=Z-PZ-ZP+PZP$ belong to $\mathcal{A}.$ So
\[
\mathcal{U}=\{T\in\mathcal{A}\text{: }QTQ=T\}=\mathcal{\{}QZQ\text{: }%
Z\in\mathcal{A}\}\subset\mathcal{A}%
\]
is an algebra and $\mathcal{U}|_{\mathcal{Y}}$ is an operator algebra on the
space $\mathcal{Y}=Q\mathcal{X},$ $\dim\mathcal{Y}<\infty.$ It is transitive,
since if $0\neq y\in\mathcal{Y}$ then $\mathcal{U}y=Q\mathcal{A}%
y=Q\overline{\mathcal{A}y}=Q\mathcal{X}=\mathcal{Y}$. As every such algebra
coincides with the algebra of all $\mathbb{D}$-linear operators (where
$\mathbb{D}=\mathbb{R},\mathbb{C}$ or $\mathbb{H}$), $\mathcal{U}%
|_{\mathcal{Y}}$ is unital. Let $R\in\mathcal{U}$ be such that
$R|_{\mathcal{Y}}=\mathbf{1}_{\mathcal{Y}}.$ Then $QRQ=R$ and $RQ=Q.$ So
$R=Q(RQ)=Q^{2}=Q.$ Thus $Q\in\mathcal{A}$. Therefore $\mathbf{1}%
=P+Q\in\mathcal{A}$, a contradiction. So $\mathcal{A}$ is unital.\bigskip
\end{proof}

3. Let $\mathcal{A}$ be a Lomonosov algebra. Is it true that all operators in
$\mathbb{D}$ are closed?\medskip

4. For which operators $S$ (respectively, representations $\pi$) the algebra
$\mathcal{A}_{S}$ (respectively, $\mathcal{A}_{\pi}$) is a maximal Lomonosov
algebra of complex (respectively, quaternion) type?

\bigskip

E. Kissin: STORM, London Metropolitan University, 166-220 Holloway Road,
London N7 8DB, Great Britain; e-mail: e.kissin@londonmet.ac.uk\medskip

V. S. Shulman: Department of Mathematics, Vologda State University, Vologda,
Russia; e-mail: shulman.victor80@gmail.com\medskip

Yu. V. Turovskii: Department of Mathematics, Vologda State University,
Vologda, Russia; e-mail: yuri.turovskii@gmail.com

\end{document}